\documentclass[a4paper]{amsart} 

\usepackage{geometry}
\usepackage{setspace}
\usepackage[utf8]{inputenc}
\usepackage[T1]{fontenc}
\usepackage[USenglish]{babel}
\usepackage{textcmds}
\usepackage{amssymb,amsmath,amsfonts,amsthm}
\usepackage{mathtools}
\usepackage[all, cmtip]{xy}
\usepackage{tikz}   
\usetikzlibrary{cd}
\usepackage{longtable}
\usepackage{enumerate}
\usepackage{array}     
\usepackage{booktabs} 
\usepackage{algorithm}
\usepackage{algpseudocode}
\usepackage{hyperref}
\usepackage{url}
\usepackage{verbatim}
\usepackage{color}

\usepackage[alphabetic,backrefs,lite]{amsrefs}
\usepackage{mathscinet}

\newtheorem{letterthm}{Theorem}

\numberwithin{equation}{section}

\newtheorem{theorem}{Theorem}[section]

\newtheorem{proposition}[theorem]{Proposition}
\newtheorem{lemma}[theorem]{Lemma}
\newtheorem{corollary}[theorem]{Corollary}
\newtheorem{claim}{Claim}

\theoremstyle{remark}

\newtheorem{remark}[theorem]{Remark}

\theoremstyle{definition}
\newtheorem{definition}[theorem]{Definition}

\newcommand{\ZZ}{\mathbb{Z}}
\newcommand{\QQ}{\mathbb{Q}}
\newcommand{\RR}{\mathbb{R}}
\newcommand{\CC}{\mathbb{C}}

\newcommand{\PP}{\mathbb{P}}

\newcommand{\Sym}{\operatorname{Sym}}
\newcommand{\de}{\partial}
\newcommand{\tm}{{\otimes m}}

\DeclareMathOperator{\GL}{GL}
\DeclareMathOperator{\PGL}{PGL}

\DeclareMathOperator{\Aut}{Aut}

\DeclareMathOperator{\Irr}{Irr}

\DeclareMathOperator{\ord}{ord}

\DeclareMathOperator{\ev}{ev}

\DeclareMathOperator{\Fix}{Fix}

\begin{document}

\title{Pluricanonical Geometry of Varieties Isogenous to a Product and Abelian Covers}

\author[M.~Alessandro, D.~Frapporti and C.~Glei\ss ner]{Massimiliano Alessandro, Davide Frapporti and Christian Glei\ss ner}
 
\address{Massimiliano Alessandro \newline  Universit\"at des Saarlandes,
Campus, Geb\"aude E2 4, D-66123 Saarbr\"ucken, Germany}
\email{malessandro@math.uni-sb.de}

\address{Davide Frapporti
\newline
Politecnico di Milano,
Dipartimento di Matematica,
Piazza Leonardo da Vinci 32,
I-20133 Milano, Italy
}
\email{davide.frapporti@polimi.it}

\address{Christian Gleissner \newline Universit\"at Bayreuth, Universit\"atsstr. 30, D-95447 Bayreuth, Germany}
\email{christian.gleissner@uni-bayreuth.de}

\subjclass[2020]{Primary:  14E20, 14L30, 14E05, 14J30,  14Q15. Secondary: 20B25, 20C15.}
\keywords{Abelian covers, finite group actions, canonical map, pluricanonical maps, birationality, varieties isogenous to a product, threefolds of general type}

\thanks{
\textit{Acknowledgments.} The authors would like to thank Roberto Pignatelli and Federico Fallucca for helpful comments and feedback. In particular, they are grateful to Roberto Pignatelli for helping with the proof of Theorem \ref{4-canonical-map}, and also for discussions concerning the content of Theorem~\ref{prop-push-decomp}. 
They are also indebted to Vladimir Lazi\'c for his help with the proof of Proposition \ref{Proposition-Fibrations}.
\newline
The authors further thank Bernhard K\"ock and Wenfei Liu for useful comments and references on the Chevalley-Weil formula.
\newline
The first-named author would like to thank Aryaman Patel and Francesco Antonio Denisi for engaging discussions and valuable insights.
\newline
The first-named and second-named authors are members of G.N.S.A.G.A. of I.N.d.A.M.
}

\begin{abstract}
We study canonical and pluricanonical maps of varieties isogenous to a product of curves, i.e., quotients of the form
\[
X = (C_1 \times \dots \times C_n)/G
\]
with $g(C_i)\ge 2$ and $G$ acting freely. For this purpose, we provide a technical result which is of general interest: a decomposition theorem for pluricanonical systems of abelian covers. This theorem provides an effective tool for the explicit study of geometric properties, such as base loci and the birationality of  pluricanonical maps.

\noindent
For threefolds isogenous to a product, we prove that the 4-canonical map is birational for $p_g \ge 5$ and construct an example attaining the maximal canonical degree for this class of threefolds. In this example, the canonical map is the normalization of its image, which admits isolated non-normal singularities. Computational classifications also reveal threefolds where the bicanonical map fails to be birational, even in the absence of genus-2 fibrations. This illustrates an interesting phenomenon similar to the non-standard case for surfaces.
\end{abstract}

\maketitle

\section*{Introduction}

Canonical and pluricanonical maps are fundamental objects in birational geometry and provide key information on the geometry of complex varieties of general type.
Recall that for a smooth projective curve $C$ of genus $g(C) \ge 2$ the canonical system $|K_C|$ defines a morphism
\[
\Phi_{K_C} \colon C \longrightarrow \mathbb{P}_\CC^{g(C)-1}.
\] 
This is an embedding if and only if $C$ is non-hyperelliptic; if $C$ is hyperelliptic, $\Phi_{K_C}$ is a finite morphism of degree two onto the rational normal curve.

In dimension two, the situation is more involved. Let $S$ be a minimal surface of general type. The canonical system $|K_S|$ need not be base-point free, i.e., the canonical map $\Phi_{K_S}$ is not always a morphism, nor is it necessarily birational onto its image. When the canonical map is generically finite, its degree is bounded from above by $36$, as shown by Beauville \cite{Beau79}.
Much effort has been devoted to constructing surfaces $S$ whose canonical map has a prescribed degree $d$ in the range $2 \le d \le 36$,  cf.~\cite{MLP23}.
A rich source of examples is provided by \emph{surfaces isogenous to a product of unmixed type} (introduced by Catanese in \cite{Cat00}), namely smooth projective surfaces of the form
\[
S = (C_1 \times C_2)/G,
\]
where $C_1$ and $C_2$ are smooth projective curves of genus at least $2$ and $G$ is a finite group acting freely and diagonally on the product. Surfaces of this form are particularly convenient for studying canonical maps for several reasons, see  \cite{GPR22}, \cite{FG23}, \cite{Fall24a}, \cite{Fall24b} (and \cite{FG20} for a similar construction in dimension three).
Indeed, many questions about $S$ reduce to questions about the curves $C_1$ and $C_2$, which are much better understood. Moreover, when $G$ is abelian, the surface can be realized as an abelian cover
\[
S \to (C_1/G) \times (C_2/G),
\]
so that techniques from the theory of abelian covers (see \cite{Par91}), combined with representation-theoretic methods, become available to analyze the canonical system in detail. In the same spirit, Catanese \cite{Cat18} used this class of surfaces to construct examples with birational canonical maps and canonical images of high degree.
This motivates our systematic study of canonical and pluricanonical maps of \emph{varieties isogenous to a product of unmixed type}, i.e., smooth projective varieties of the form
\[
X = (C_1 \times \dots \times C_n)/G,
\]
where each $C_i$ is a smooth projective curve of genus at least $2$ and $G$ is a finite group acting freely and diagonally on the product. These objects naturally generalize surfaces isogenous to a product of unmixed type to higher dimensions; for brevity, we will often refer to them as \textit{VIPs}.

Our main focus is to investigate the birationality of their associated canonical and pluricanonical maps. 
When attempting to extend to VIPs and their pluricanonical maps the above-mentioned strategy, significant challenges arise.
Recall that the Chevalley-Weil formula \cites{CheWeil34,Weil35} (see also \cite{EL80}, \cite{K04}, \cite{K05}, \cite{LL25a}, \cite{LL25b}  for various generalizations) allows one to determine the characters of the canonical and pluricanonical representations of a Riemann surface \( C \) endowed with an action of a finite group \( G \). While this information is useful for the study of the \( m \)-canonical linear system \( H^0(X, K_X^{\otimes m}) \), it is by itself insufficient to extract finer geometric properties, such as the presence of base points.
To overcome this limitation in the case where $G$ is abelian, we use Pardini's theory of abelian covers to obtain a complete and detailed decomposition of the pluricanonical systems, rather than just the character of the Galois action.
Such decompositions have been studied explicitly in the literature for \( m = 1 \) by Pardini \cite{Par91} and Liedtke \cite{Lie03}, and for \( m = 2 \) by Bauer–Pignatelli \cite{BP21}. Since we could not find a reference in the  literature for $m\geq 3$, we provide a general version of such a decomposition formula for finite abelian covers and arbitrary \( m \), which we believe to be of independent interest. This is the content of the following theorem.\footnote{
In a private communication, Roberto Pignatelli informed us that a similar version of Theorem \ref{thm-decomp} was presented by him in a talk at the University of Bayreuth in 2022. We point out that our proof is different and explicitly provides the correction terms $t_{(H,\psi)}^{r(m,H,\psi, \chi)}$. This explicit description is important for some of our computations, for instance to study the base loci of the pluricanonical systems.
}

\begin{letterthm}[Theorem \ref{prop-push-decomp}] \label{thm-decomp}
Let $\pi \colon X \to Y$ be a finite abelian cover where $X$ is normal and $Y$ is smooth with Galois group $G$. Let $m$ be a non negative integer,  
$\chi\in \Irr(G)$ a character, $H\leq G$ a cyclic subgroup  and $\psi$ a generator of $\Irr(H)$. 
Denote by  $D_{(H, \psi)}$ the sum of all components of the branch divisor of $\pi$  with stabilizer group $H$
and character $\psi$.
Then there exists a unique integer  $0\leq k(H,\psi, \chi)\leq \vert H \vert -1$ such that $\chi_{\vert H}=\psi^{k(H,\psi, \chi)}$.
Moreover, define 
\[
    r(m,H,\psi, \chi):=k(H,\psi, \chi)-m + \left\lceil \frac{m-k(H,\psi, \chi)}{\vert H \vert }\right\rceil \vert H \vert
\]
and
\[
    \mu(m,H,\psi, \chi):= m-\left\lceil \frac{m-k(H,\psi, \chi)}{\vert H \vert }\right\rceil.
\]
Then there exists an isomorphism of line bundles as follows
\begin{equation*}
    \pi_\ast(K_X^{\otimes m})^\chi\cong  \mathcal{O}_Y\left(\sum_{(H,\psi)}  \mu(m,H,\psi, \chi) \ D_{(H,\psi)}\right)\otimes K_Y^{\otimes m} \otimes L_\chi^{-1},
\end{equation*}
where $L_\chi^{-1}:=(\pi_\ast \mathcal{O}_X)^\chi$. The $m$-canonical system has the following decomposition
\begin{equation*}
    H^0(X,K_X^{\otimes m})=\bigoplus_{\chi\in \Irr(G)} \prod_{(H,\psi)} t_{(H,\psi)}^{r(m,H,\psi, \chi)}
    \ \pi^\ast H^0\left(Y,\mathcal{O}_Y\left(\sum_{(H,\psi)}     \mu(m,H,\psi, \chi) \ D_{(H,\psi)}\right)\otimes K_Y^{\otimes m} \otimes L_\chi^{-1}\right),
\end{equation*}
where  $\{t_{(H,\psi)}=0\}$ is a local equation of the ramification divisor over $D_{(H,\psi)}$.
\end{letterthm}

Next, we give a generalization to higher dimensions and to pluricanonical maps of the representation-theoretic birationality criteria (see Proposition \ref{proj-embedding-of-G}, Corollary \ref{corollary-necessary-bir-criterion}, Proposition \ref{birationality-criteria} and Corollary \ref{birational-criteria-Y=P1^n}) originally introduced by Catanese \cite{Cat18} in the surface case for the canonical map.  

These criteria togheter with the Chevalley-Weil formula (see Theorem \ref{Chev-Weil-Pluri}) and Theorem \ref{thm-decomp} provide us with the technical tools to analyze the  pluricanonical maps and systems of an unmixed VIP with  abelian Galois group in concrete examples. 
More precisely, we implement a classification algorithm for regular unmixed three-dimensional VIPs with  abelian Galois group.  Our approach combines modifications of the algorithms presented in \cite{FG16} and \cite{GPR22} with implementations of the above-mentioned tools: given an integer $\chi\leq -1$, our algorithm classifies all unmixed three-dimensional VIPs with holomorphic Euler-Poincaré characteristic $\chi$ and with abelian group $G$ such that the $G$-action on each curve is faithful. 
We run it for $\chi$ in the range $-8 \leq \chi \leq -1$. The output shows that no threefolds in the range $-7 \leq \chi \leq -1$ satisfy the above-mentioned birationality criteria for the canonical map. However, examples with a birational canonical map do exist for $\chi=-8$. We analyze in detail one such example, which leads to the following theorem.

\begin{letterthm}[Theorem \ref{3-fold-normalization-map}]\label{thm-normalization}
    There exists a regular smooth projective threefold $X$ whose canonical map $\Phi_1\colon X\to \Sigma \subset \PP^{15}$ is the normalization map of its image. In particular, $\Sigma$ has degree $K^3_X=-48\cdot\chi(\mathcal{O}_X)=384$ and at least one isolated non-normal singularity.
\end{letterthm}

Note that the degree of the image of the canonical map for a three-dimensional VIP $X$ is bounded from above by $-48\cdot \chi(\mathcal{O}_X)$. Hence, the previous theorem shows that this bound is sharp. 
Moreover, the existence of a threefold whose canonical image has an isolated
non-normal singularity is a noteworthy phenomenon, as we were unable to find any such
examples in the literature.

We also run a classification algorithm which allows non-faithful actions of the group $G$ on the curves $C_i$. Applying it for $\chi=-1$ and $|G|\leq 16$, we obtain $347$ families of threefolds, summarized in Table \ref{Table-chi-1}.
The table lists the Galois group $G$, the orders of the kernels, the branching types of the Galois covers $C_i \to C_i/G_i \simeq \mathbb{P}^1$ and the Hodge numbers of the threefolds. Additionally, for the canonical and bicanonical maps, it indicates the number of families for which these maps are base-point free, birational, non-birational and the number of those for which we cannot decide about the birationality  by using our criteria. 
We also comment on the failure of birationality of the bicanonical map in our examples. To put things into perspective, we  recall the situation in lower dimensions.
For a curve $C$, the bicanonical map is always an embedding unless $g(C)=2$; in this case, it is a double cover of a conic in $\mathbb{P}^2$. In dimension two, the works of Bombieri~\cite{Bomb73} and Reider~\cite{Reider88} imply that the bicanonical map of a minimal surface $S$ of general type with $K_S^2 \geq 10$ is birational, except when $S$ admits a genus-2 fibration. 
This is the so-called \emph{standard case}; see also the survey \cite[Section 2]{BCP06} for an overview on the topic.
In our setting, the bicanonical map of some threefolds listed in Table~\ref{Table-chi-1} fails to be birational for the presence of a genus-2 fibration: this is similar to the standard case. However, we also identify and highlight  examples without any genus-2 fibration whose bicanonical map is nevertheless non-birational, see Remark \ref{Remark-Table}. This phenomenon is particularly interesting, as it is similar to the \emph{non-standard case} for surfaces.

We further observe that for all examples listed in Table \ref{Table-chi-1},  $K_X^{\otimes 2}$ is base-point free and $K_X^{\otimes 3}$ defines a birational map. This raises the natural questions of whether these properties hold for all regular unmixed three-dimensional VIPs.

\medskip

Besides constructing and analyzing explicit examples, we also establish general results concerning the (pluri)canonical maps of varieties isogenous to a product, which are valid even when the group action is non-diagonal.
In particular, we prove a birationality result for the $4$-canonical map of threefolds isogenous to a product. Our argument relies on the work of Chen-Hu \cite{CH21} and on the geometry of Horikawa $(1,2)$-surfaces \cite{Hor76II}. Recall that, by \cite{CCZ07}, the $m$-canonical map of a minimal Gorenstein projective threefold of general type with canonical singularities is birational for $m \geq 5$.

\begin{letterthm}[Theorem \ref{4-canonical-map}] \label{thm-4-can}
Let $X$ be a threefold isogenous to a product of curves. If $p_g(X) \geq 5$, then the $4$-canonical map of $X$ is birational.
\end{letterthm}

\medskip

The paper is structured as follows.

\smallskip

In Section \ref{Sect2}, we review the basic theory of varieties isogenous to a product of curves. We also establish general results concerning the pluricanonical maps of VIPs and prove  Theorem \ref{thm-4-can}. 

In Section \ref{Pluri-Reps-Sec}, we study the characters of the pluricanonical representations of a VIP and give a proof of the Chevalley-Weil formula  for $m$-canonical representations with $m\geq 2$ (see Theorem \ref{Chev-Weil-Pluri}).

In Section \ref{secA}, we focus on abelian covers. Here, we prove Theorem \ref{thm-decomp} and then apply it to describe the pluricanonical systems of unmixed VIPs with abelian Galois group.

In Section \ref{section-with-diagram}, we discuss Catanese's birationality criteria, adapting them to the context of pluricanonical maps and higher-dimensional varieties.

In Section \ref{FinalSection}, we describe in detail the example leading to Theorem \ref{thm-normalization}. Finally, we present and discuss the above-mentioned computational results, see Table \ref{Table-chi-1} and Remark \ref{Remark-Table}.

\section{Varieties Isogenous to a Product }\label{Sect2}

We recall the definiton of a variety  isogenous to a product.

    \begin{definition}
        A complex variety X is said to be \emph{isogenous to a product of curves} (for short \emph{VIP}) if it is isomorphic to a quotient
        \[
        X \cong (C_1 \times \dotso \times C_n)/G,
        \]
        where each $C_i$ is a smooth projective curve of genus $g_i\geq 2$ and  $G\subset \Aut(C_1 \times \dotso \times C_n)$ is a finite group acting  freely. 
       
    \end{definition}

    Throughout this paper we mainly focus on the case where $G\subset \Aut(C_1) \times \dotso \times \Aut(C_n) $, i.e., the action of $G$ is \emph{of unmixed type}. Hence, the composition with the projection on the $i$-th  factor provides an action  $\psi_i \colon G \to \Aut(C_i)$ which is not necessarily injective, and thus might have a non-trivial kernel $ K_i$. The quotient group $G/K_i$, which acts faithfully on $C_i$, is denoted by $G_i$.

    \begin{remark}\label{minimal-realization}
    In \cite{FGR26} (and in the appendix to \cite{Rito5} for the mixed case) the authors show that every variety $X$ isogenous to a product has a unique realization $ (C_1 \times \dotso \times C_n)/G$ such that  
    \[
     K_1 \cap \dotso \cap \widehat{K_i} \cap \dotso \cap K_n=\{1\} \qquad \makebox{for all} \quad  i=1, \dotso, n.
    \]
    More precisely, if $(D_1 \times \dotso \times D_n)/G'$  is another realization of $X$, then, up to a permutation of the curves $D_i$, there exist biholomorphic maps $f_i\colon C_i \to D_i$ such that 
    \[
    (f_1, \dotso, f_n) \cdot G \cdot (f_1, \dotso, f_n)^{-1}=G'.
    \]
    
    This realization is called \emph{the minimal realization of $X$}.  
    \end{remark}
    
    From now on, when we consider a realization $\big(C_1 \times \dotso \times C_n\big)/G$ of a variety isogenous to a product $X$, we assume it is minimal.

    \subsection{Group-Theoretic Description of an Unmixed VIP}
   
    Let $X\cong (C_1 \times \dotso \times C_n)/G$ be a variety isogenous to a product of unmixed type.  Up to isomorphism, we can uniquely attach to $X$  a group $G$ together with curves $C_i$ and actions 
     \[
     \psi_i\colon G \to \Aut(C_i) \quad \makebox{with kernels} \quad K_i.
     \]
    We can describe the induced faithful actions 
    \[
    \overline{\psi_i}\colon G_i=G/K_i \hookrightarrow \Aut(C_i)
    \]
    thanks to \emph{Riemann's existence theorem} (cf.~\cite[Sections III.3 and III.4]{Mir95}). For this purpose we recall the notion of a generating vector.

\begin{definition}\label{DefGV}
Let $2 \leq n_1 \leq \ldots \leq n_r$ and  $g' \geq 0$ be integers and $H$ be a finite group.  
A \textit{generating vector} for $H$  of type $[g';n_1, \ldots ,n_r]$ is a $(2g'+ r)$-tuple 
$(d_1,e_1, \ldots , d_{g'},e_{g'};h_1, \ldots , h_r)$ of elements of $H$ such that: 
\begin{itemize}
\item[i)]
$ H = \langle d_1,e_1, \ldots , d_{g'},e_{g'},h_1, \ldots , h_r\rangle$,
\item[ii)]
$\displaystyle{\prod_{i=1}^{g'}[d_i,e_i]\cdot h_1 \ldots  h_r=1_H}$,
\item[iii)]
$\mathrm{ord}(h_i)= n_i$  for all $1 \leq i \leq r$. 
\end{itemize}
\end{definition}

\begin{theorem}[Riemann's Existence Theorem]\label{RsET}
A finite group $H$ acts faithfully and  biholomorphically on some compact Riemann surface $C$ of genus $g(C)$  if and only if  there exists a generating vector 
\[
V:=(d_1,e_1, \ldots , d_{g'},e_{g'};h_1, \ldots , h_r)
\]
for  $H$ of type $[g';n_1,\ldots,n_r]$,
such that  the  Hurwitz formula holds:
\begin{equation*} 
2g(C) -2 = |H| \bigg( 2g'-2+\sum_{i=1}^r \frac{n_i - 1}{n_i} \bigg). 
\end{equation*}

In  this case $g'$ is the genus of the quotient Riemann surface $C' := C/H$
and the $H$-cover $C\to C'$ is branched in $r$ points $\{x_1, \ldots , x_r\}$ with
branching indices $n_1, \ldots , n_r$, respectively.
Moreover, the cyclic groups $\langle h_1\rangle, \ldots , \langle h_r\rangle$ and
their conjugates provide the non-trivial stabilizers of the action of $H$ on $C$.
\end{theorem}

According to the previous theorem, 
the set 
\[
\Sigma_{V}:= 
\bigcup_{h \in H} \left(
h\cdot \langle h_1 \rangle \cdot  h^{-1} \cup 
\dotso
\cup
h\cdot \langle h_r \rangle \cdot  h^{-1}
\right) 	
\]
is the set of group elements of $H$ with non-empty fixed locus on the curve $C$. It is hence called the \emph{stabilizer set of $V$}.

\begin{definition} \label{algebraicdata}
To a variety $X$ isogenous to a product of unmixed type we attach the tuple
\[(G, K_1, \dotso, K_n, V_1, \dotso, V_n)
\]
and  call it an  \emph{algebraic datum} of $X$. Here, $V_i$ is a generating vector for the action $\overline{\psi_i} \colon G_i \to \Aut(C_i) $.

\end{definition}

\begin{remark}\label{freeness-condition}

Note that 
$\bigcap_{i=1}^n\Sigma_{V_i}\cdot K_i =\{1_G\}$ since the group $G$ acts freely on the product $C_1\times \dotso \times C_n$.

\end{remark}

So far we attached to a given a variety $X$
isogenous to a product of unmixed type the group theoretical information provided by an algebraic datum.
Conversely, let $G$ be a finite group and let $K_i \trianglelefteq G$ be normal subgroups such that
\[
     K_1 \cap \dotso \cap \widehat{K_i} \cap \dotso \cap K_n=\{1\} \qquad \makebox{for all} \quad  i=1, \dotso, n.
\]
Let $(V_1,\dotso,V_n) $ be an $n$-tuple of generating vectors, where $V_i$ is a generating vector for $G_i:=G/K_i$.
Assume moreover that the freeness condition of Remark \ref{freeness-condition} holds.
Then, by Riemann's existence theorem, there exists a variety $X$ isogenous to a product of unmixed type with $(G, K_1, \dotso, K_n, V_1, \dotso, V_n)
$ as an algebraic datum.

\subsection{Generalities on Pluricanonical Maps of a VIP} \ 

In this section, we collect some general results on the $m$-canonical map of a VIP.
Our main goal  is to study its birationality for a given $m\geq 1$.  Since such a variety is of general type, there exists a positive integer $m_0$ such that this map is birational for all $m\geq m_0$.

\medskip

Let $X \cong (C_1 \times \dots \times C_n)/G$ be a variety isogenous to a product. We denote by $Z$ the product $C_1\times \dotso \times C_n $. In order to avoid ambiguity, in this subsection we write $\Phi_m^Y$ for the $m$-canonical map of a variety $Y$.  Since
$$H^0(X,K_X^\tm)\cong H^0(Z,K_Z^\tm)^G\subset H^0(Z,K_Z^\tm), $$
we can complete a basis $\{s_0, \dotso, s_k\}$ of $H^0(X,K_X^\tm) $ to a basis $\{s_0, \dotso, s_k, t_0, \dotso, t_l\} $ of $H^0(Z,K_Z^\tm)$. 
We observe that the $m$-canonical map $\Phi_m^Z\colon Z\to \PP(H^0(Z,K_Z^\tm)^\vee)$ is a finite morphism which factors through the product map $\Phi_m^{C_1}\times \dotso \times \Phi_m^{C_n}$ and the Segre embedding. Thus, we have a diagram as follows \begin{equation}\label{pluricanonical-diagram}
        \xymatrix{
        Z \ar[r]^\Xi \ar[d]_{\Phi_m^Z} & X\ar@{-->}[d]^{\Phi_m^X}\\
         \PP(H^0(Z,K_Z^\tm)^\vee) \ar@{-->}[r]^p & \PP(H^0(X,K_X^\tm)^\vee)
        }
\end{equation}
where the map  $p$ is the linear projection onto the first $k+1$ coordinates, namely
\[
    p(x_0 \colon  \dotso\colon  x_k\colon y_0 \colon \dotso \colon y_l) = (x_0 \colon  \dotso\colon  x_k).
\]
    Let $Z_m:=\Phi_m^Z(Z)$ be the $m$-canonical image of $Z$. 
   For the centre of the projection $p$ the following holds
    \begin{equation}\label{proj-centre}
         Z_m\cap V(x_0, \dotso, x_k)=\Phi_m^Z\left(V(s_0,\dotso, s_k)\right).
    \end{equation}

\begin{proposition}\label{bpfreeness}
   If the $m$-canonical system $\vert K_X^\tm \vert$ is basepoint free, then $\Phi_m^X\colon X \to X_m\subset \PP^{P_m-1}$ is finite. 
\end{proposition}

\begin{proof}
    Note that $K_X^\tm$ being basepoint free amounts to saying that $V(s_0,\dotso, s_k)=\emptyset$. This is equivalent by \eqref{proj-centre} to $Z_m$ not meeting the centre of the projection $p$. In other words, $\Phi_m^X$ is a morphism if and only if $p_{\vert Z_m}$ is a morphism. In this case  $p_{\vert Z_m}$ is a finite morphism, therefore from diagram \eqref{pluricanonical-diagram} it follows that  $\Phi_m^X$ must be finite.
\end{proof}

\begin{corollary}
    If $\Phi_m^X $ is birational and the $m$-canonical system $\vert K_X^\tm \vert$ is basepoint free, then $\Phi_m^X $ is indeed the normalization map of $X_m$.
\end{corollary}

\begin{proposition}\label{hyperellBir}
\begin{enumerate}
    \item   If one of the curves $C_i$ is hyperelliptic, then $\Phi_1^X$ is not birational.

    \item If one of the curves $C_i$ has genus $2$, then $\Phi_2^X$ is not birational.
\end{enumerate}
\end{proposition}

\begin{proof}\
    (1) Suppose w.l.o.g.~that $C_1$ is hyperelliptic. Then the map $\Phi_1^{C_1}\colon C_1 \to \PP^1$ is a double cover, thus $Z_1\cong \PP^1\times \overset{\sim}{C_2} \times \dotso \times \overset{\sim}{C_n} $. If  $\Phi_1^X$ is birational, then from diagram \eqref{pluricanonical-diagram} we get a dominant rational map 
    \[
    \PP^1\times \overset{\sim}{C_2} \times \dotso \times \overset{\sim}{C_n} \dashrightarrow X,
    \]
    a contradiction since $\kappa( \PP^1\times \overset{\sim}{C_2} \times \dotso \times \overset{\sim}{C_n}) =-\infty$.

    (2) Suppose w.l.o.g. that $g(C_1)=2$. Then the map $\Phi_2^{C_1}\colon C_1 \to Q\subset\PP^2$ is double cover of a conic $Q$, thus $Z_2 \cong \PP^1\times \overset{\sim}{C_2} \times \dotso \times \overset{\sim}{C_n}$. We immediately get the same  contradiction as in (1). 
\end{proof}

\begin{proposition}\label{embedding-SIP}
    Let $S\cong(C_1\times C_2)/G$ be a surface isogenous to a product. Then the $m$-canonical map $\Phi_m^S$  is an embedding for $m\geq 3$.
\end{proposition}

\begin{proof}
    Since $K_S$ is ample and $K^2_S=8\chi(\mathcal{O}_S)\geq  8$,  the result follows by 
      \cite[Theorem VII.5.1]{B-H-P-V}.
\end{proof}

\begin{remark}

By  \cite[Theorem 1.1]{CCZ07} the $m$-canonical map of a minimal Gorenstein projective threefold of general type with
canonical singularities is birational for $m\geq 5$. This holds in particular for  threefolds isogenous to a product since they are smooth and their canonical bundle is ample.
\end{remark}

Now we investigate the birationality of $\Phi_4^X$ for a threefold $X$ isogenous to a product. For this purpose, we recall that a $(1,2)$ surface $S$ is a nonsingular projective surface of general type whose minimal model $S'$ has $K^2_{S'}=1$ and $p_g(S')=2$. By \cite[Theorem 4.8]{Hor76II} a $(1,2)$ surface is simply connected.

\begin{theorem} \label{4-canonical-map}
     Let  $X$  be a threefold isogenous to a product. If $p_g(X)\geq 5$, then the $4$-canonical map is birational.
\end{theorem}

\begin{proof}
    Since $p_g(X)\geq 5$, \cite[Theorem 1.1]{CH21} implies that  $\Phi_4^X$ is birational if and only if there is no smooth threefold $Y$ birational to $X$ admitting a fibration $Y\to B$, where $B$ is a smooth curve and the general fibre $S$ is a  $(1, 2)$-surface. 
    By contradiction, assume that such a threefold $Y$ exists. Denote by $S'$ the birational image  in $X$ of the general fibre $S\subset Y$ and by $\hat{S}$ a resolution of singularities of $S'$.    Notice the
    $\hat{S}$ is simply connected since it is 
    birational to the  $(1,2)$ surface $S$. Hence, the map $ \hat{S} \to S'\subset  X$ lifts to the universal cover of $X$, which is the product $\Delta^3$ of three unit disks:
    \[
    \xymatrix{
    \Delta^3 \ar[r] & X \\
     & \hat{S} \ar[u] \ar[ul]
    }
    \]
    The map  $\hat{S}\to \Delta^3$ must be constant by the maximum principle, a contradiction.
\end{proof}

\section{Pluricanonical Representations of an Unmixed VIP} \label{Pluri-Reps-Sec}

Given a variety $X \cong (C_1 \times \dotso \times C_n)/G$ isogenous to a product of unmixed type, we always have a diagram as follows
\begin{equation}\label{VIP-diagram}
    \xymatrix{
    C_1 \times \dotso \times C_n \ar[rr] \ar[dr] &  &  C_1' \times \dotso \times C_n'\\
         & X \ar[ur]_\pi  & \\
    }
\end{equation}
Here, the horizontal map is the product of the quotient maps  $C_i \to C_i':=C_i/G_i$. Since the action of $G$ on $C_1\times \dotso \times C_n$ is faithful, the natural map 
    \[
     G \to G_1\times \dotso  \times G_n, \quad  g \mapsto ([g]_1, \dotso , [g]_n),
    \]
    is injective, and therefore $G$ can be considered as a subgroup of $G_1\times \dotso  \times G_n$. 
    \begin{remark}
        Note that, since $(C_1 \times \dotso \times C_n)/G$ is  the minimal realization of $X$,  the group $G$ even embeds in  $G_1\times \dotso \times \widehat{G_i} \times \dotso \times  G_n$.
    \end{remark}
   
    The Galois group of the map $\pi$ in the diagram above is the quotient of the normalizer $ N_{G_1\times \dotso \times G_n}(G)$ by $G$.
    Thus, the map $\pi$ is Galois if and only if $G$ is normal in $G_1\times \dotso \times G_n$.
    This is clearly the case if $G$ is abelian. Conversely, we have: 
    \begin{proposition}\label{G-normal}
        If $G$ is normal in $G_1\times \dotso \times G_n$, then $G$ is abelian.
    \end{proposition}
    
    \begin{proof}
    Suppose $G$ is normal. Then, for $g,g_1\in G$ there exists $h\in G$ such that 
    \begin{equation}\label{commutator-relation}
         (g_1,1,\dotso,1) \cdot (g,g,\dotso ,g) \cdot (g_1^{-1},1,\dotso, 1)\cdot(h^{-1},h^{-1}, \dotso, h^{-1})\in K_1\times \dotso \times K_n. 
    \end{equation}
    In other words, by the minimality of the realization we have that
    $ gh^{-1}\in K_2\cap \dotso \cap K_n=\{1\}.$
    Thus, $g=h$, which implies by \eqref{commutator-relation} that
    $[g_1,g]\in K_1$. Since $g_1,g\in G$ are arbitrary, this shows that $[G,G]\leq K_1$. Similarly, one can show that $[G,G]\leq K_i$, $i=2,\dotso, n$. Thus, we get that $[G,G]=\bigcap_{i=1}^n K_i=\{1\}$, i.e. $G$ is abelian.
    \end{proof}

Let $m$ be a positive integer. We observe that  the action of the Galois group $N_{G_1\times \dotso \times G_n}(G)/G$ of the map $\pi$ on $X$ induces a representation by pullback of $m$-canonical forms, namely
\[
\rho_{m}\colon N_{G_1\times \dotso \times G_n}(G)/G \to \GL(H^0(X,K_X^\tm)), \qquad g \mapsto \ \left \{ \omega \mapsto (g^{-1})^\ast \omega \right \},
\]
which we call \emph{the $m$-canonical representation of the cover $\pi$.} Its character is denoted by $\chi_m$.

Since 
$$H^0(X,K_X^\tm)\cong H^0(C_1\times \dotso \times C_n, K_{C_1\times \dotso \times C_n}^\tm)^G,$$
the next lemma tells us that 
we can compute the character $\chi_m$ of $\rho_m$ in terms of the character of the representation 
\begin{equation}\label{rho}
    \rho \colon G_1 \times \dotso \times G_n\to \GL\big(H^0(C_1\times \dotso \times C_n, K_{C_1\times \dotso \times C_n}^\tm)\big).
\end{equation}

\begin{lemma}\label{lemma-normalizer}
    Let $\rho \colon G \to \GL(V)$ be a representation of a finite group $G$. Given a subgroup $H\leq G$, we denote by $N$ its normalizer in $G$. 
    \begin{enumerate}
        \item The vector space  $V^H$ is $N$-invariant.

        \item Let  $\chi_{\rho_{|N}}=\sum_{\chi \in \Irr(N)} n_\chi \chi  $ be the character of the restriction $\rho_{|N} $. Then the character of the representation 
        $$\overline{\rho_{|N} } \colon N/H \to \GL(V^H)
         \qquad \makebox{is} 
         \qquad 
    \chi_{\overline{\rho_{|N} }}=\sum_{\stackrel{\chi \in \Irr(N)}{\chi_{|H}=\chi(1)\chi_{triv}}} n_\chi \chi.
    $$
    \end{enumerate}

\end{lemma}

\begin{proof}
    As for (1), let $v\in V^H$, $g\in N$ and  $h\in H$. Then $g^{-1} h g\in H$, and therefore   $\rho(g^{-1}hg)(v)=v$, which implies
    $ \rho(h)\rho(g)(v)=\rho(g)(v)$, i.e. $\rho(g)(v)\in V^H$.
   
    As for (2), by Maschke's lemma there exists an $N$-invariant complement $W$ of $V^H$, namely $V=V^H\oplus W$. 
    A constituent $n_\chi\chi$ of $\chi_{\rho_{|N}}$
    belongs to the character of $V^H$ if and only if $\chi_{|H}$ is a multiple of the trivial character $\chi_{triv}$.
\end{proof}

By Künneth formula the action $\rho$ from \eqref{rho} is the tensor product of the representations $$ G_i \to \GL(H^0(C_i, K_{C_i}^\tm),$$
whose characters can be computed using the Chevalley-Weil formula (see \cites{CheWeil34,Weil35}) in terms of the generating vectors $V_i$ of an algebraic datum of $X$. Thus, by the Chevalley-Weil formula and by Lemma \ref{lemma-normalizer} we recover from an algebraic datum of $X$ the character $\chi_m$ of the representation $\rho_m$.

\subsection{The Chevalley-Weil Formula for Pluricanonical Representations} \label{subsec-Chev-Weil}

Let $C$ be a compact Riemann surface and
$G$ be a finite group with a faithful action
\[
\psi\colon G \rightarrow \Aut(C).
\]

Let $\varphi_m$ be the representation of $G$ induced by pullback of $m$-canonical forms, namely
\begin{equation*}
\varphi_m\colon G \to \mathrm{GL}\left( H^0(C, K_C^\tm ) \right)\,,  \quad  g \mapsto [ \omega \mapsto \psi(g^{-1})^{\ast}\omega],
\end{equation*}
and  let $\chi_{\varphi_m}$ be its character. 
The character $\chi_{\varphi_m}$ has the decomposition
\[
\chi_{\varphi_m} = \sum_{\chi \in \Irr(G)} \langle \chi, \chi_{\varphi_m}\rangle \cdot \chi,
\]
where $\langle \cdot, \cdot \rangle$ denotes the hermitian inner product in the space of class functions of the group $G$.

The Chevalley-Weil formula provides a way to compute the coefficients 
$\langle \chi, \chi_{\varphi_m} \rangle$, hence the character $\chi_{\varphi_m}$, in terms of a generating vector $(d_1,e_1, \ldots , d_{g'},e_{g'};h_1, \ldots , h_r)$  of type $[g';n_1, \ldots ,n_r]$ associated with the cover $F\colon C \to C/G$.

\begin{theorem}[Chevalley-Weil Formula for Pluricanonical Representations, cf. \cites{CheWeil34,Weil35}] \label{Chev-Weil-Pluri}

Let $\chi_{\varphi_m}$ be the character of the  $m$-canonical representation
\[
\varphi_m\colon G \to \GL\big(H^0(C, K_C^\tm)\big) 
\]
and let $\chi \in \Irr(G)$ be the character of an irreducible representation $\varrho \colon G \to \mathrm{GL}(V)$. If $m\geq 2$, then it holds 

\begin{equation*}
\langle \chi,\chi_{\varphi_m} \rangle= \frac{2m}{\vert G \vert}\chi(1_G)(g(C)-1) - \chi(1_G)(g(C/G)-1)- \sum_{i=1}^r\sum_{\alpha=0}^{n_i-1} N_{i,\alpha} \frac{[-\alpha-m]_{n_i}}{n_i}.
\end{equation*}
Here $[b]_{n_i} \in \lbrace 0, \ldots, n_i-1\rbrace$ is the congruence class of the integer $b$ modulo $n_i$ and $N_{i,\alpha}$ denotes the multiplicity of $ \zeta_{n_i}^\alpha$ as an eigenvalue of  $\varrho(h_i)$. 
\end{theorem}

\begin{proof}
For all $ g \in G\setminus \{1\}$ we consider the Eichler trace formula \cite[p. 281]{Farkas-Kra}
\[
\chi_{\varphi_m}(g)=\sum_{p\in \Fix(g)}\frac{J_p(g^{-1})^{\nu_g}}{1-J_p(g^{-1})},
\]
where $J_p(g^{-1})$ denotes the derivative of $g^{-1}$ at $p$ and $\nu_g\in \{1,\dotso, \ord(g)\}$ is the unique integer such that  $m=\mu_g \ord(g)+\nu_g$. 
By definition of $\langle \cdot, \cdot \rangle$, we obtain 
$$
\begin{array}{rcl}
\langle \chi,\chi_{\varphi_m} \rangle =
{\displaystyle \frac{1}{|G|}\sum_{g \in G}\chi(g) \overline{\chi_{\varphi_m}(g)}} 
&=&
{\displaystyle \chi(1_G)\frac{(2m-1)(g(C)-1)}{\vert G \vert}}+
{\displaystyle \frac{1}{|G|}\sum_{g \ne 1_G}\chi(g) \overline{\chi_{\varphi_m}(g)}} 
\\
&=& 
{\displaystyle \chi(1_G)\frac{(2m-1)(g(C)-1)}{\vert G \vert}}+
{\displaystyle \frac{1}{|G|}\sum_{\substack{g \in G \\ g \ne 1_G}} 
\sum_{ p \in \Fix(g)}\chi(g) \frac{J_p(g)^{\nu_g}}{1 -  J_p(g)}}
\\
&= &
{\displaystyle \chi(1_G)\frac{(2m-1)(g(C)-1)}{\vert G \vert}}+
{\displaystyle \frac{1}{|G|}\sum_{p \in \Fix(C) } 
\sum_{\substack{g \in G_p \\ g \ne 1_G}} \chi(g) \frac{J_p(g)^{\nu_g}}{1 -  J_p(g)}}\,,
\end{array}
$$
where 
\[
\Fix(C):=\big\lbrace p\in C ~ \big\vert ~ G_p\neq \lbrace 1_G \rbrace \big\rbrace.
\] 
Let $\{q_1,\ldots, q_r\} $ be the 
branch locus of $F\colon C \to C/G$. 
For all $1 \leq i \leq r$ there exists a point  $p_i \in F^{-1}(q_i)$ with 
$G_{p_i}= \langle h_i \rangle$  and  for each $h \in G$ it holds
\[
G_{h(p_i)}=\langle h h_i h^{-1}\rangle.
\]
Moreover, every $p \in \Fix(C)$ maps to a branch point of $F$.
Since $h_i$ is the local monodromy, we have 
\[
J_{h(p_i)}(hh_ih^{-1})=\zeta_{n_i}
\]
and we conclude that the Jacobian  $J_{h(p_i)}(hh_i^l h^{-1})$ is equal to $\zeta_{n_i}^l$ for all $l \in \mathbb Z$. 
Since $\chi$ is a class function we also have $\chi(h_i^l)= \chi(hh_i^lh^{-1})$. 
This implies 
\[
{\displaystyle \frac{1}{|G|}\sum_{p \in \Fix(C) } 
\sum_{\substack{g \in G_p \\ g \ne 1_G}} \chi(g) \frac{J_p(g)^{\nu_g}}{1 -  J_p(g)}}=
\frac{1}{|G|} \sum_{i=1}^r \frac{|G|}{n_i} \sum_{l=1}^{n_i-1} \chi(h_i^l)\frac{\zeta_{n_i}^{l\nu_{h_i^l}} }{1 - \zeta_{n_i}^{l}}=
\sum_{i=1}^r \frac{1}{n_i} \sum_{l=1}^{n_i-1} \frac{\sum_{\alpha = 0}^{n_i-1} 
N_{i,\alpha}\cdot \zeta_{n_i}^{l(\alpha+\nu_{h_i^l})}}
{1 - \zeta_{n_i}^{l}}.
\]	
To simplify the last term in the chain of equalities above, we write: 
\[
\sum_{l=1}^{n_i-1} \frac{ (\zeta_{n_i}^{l})^ {\alpha+\nu_{h_i^l}}}{1 - \zeta_{n_i}^{l}} = 
\sum_{l=1}^{n_i-1} \frac{ (\zeta_{n_i}^{l})^ {-\alpha-\nu_{h_i^l}}}{1 - \zeta_{n_i}^{-l}} = 
\sum_{l=1}^{n_i-1} \frac{ (\zeta_{n_i}^{l})^ {-\alpha-m}}{1 - \zeta_{n_i}^{-l}} =
\sum_{l=1}^{n_i-1} \frac{ (\zeta_{n_i}^{l})^ {[-\alpha-m]_{n_i}}}{1 - \zeta_{n_i}^{-l}}
\]
By \cite[Eq. 8.8]{Reid85} we have
\[
\sum_{l=1}^{n_i-1} \frac{ (\zeta_{n_i}^{l})^ {[-\alpha-m]_{n_i}}}{1 - \zeta_{n_i}^{-l}}=\frac{n_i -1}{2} - [-\alpha-m]_{n_i}.
\]
In summary,
\[
\langle \chi,\chi_{\varphi_m} \rangle  =
{\displaystyle \chi(1_G)\frac{(2m-1)(g(C)-1)}{\vert G \vert}}+
 \sum_{i=1}^r \sum_{\alpha=0}^{n_i-1} \frac{N_{i,\alpha}}{n_i}\bigg(\frac{n_i-1}{2} -[-\alpha-m]_{n_i}\bigg).
\]
Hurwitz' formula and the identity $\sum_{\alpha =0}^{n_i-1} N_{i,\alpha}=\chi(1_G)$ yields the  Chevalley-Weil formula. 
\end{proof}

\begin{remark}
    We want to point out that in order to determine the integers $N_{i,\alpha} $ it is not necessary to know the representation $\varrho\colon G \to \GL(V)$. Indeed, it suffices to know the character $\chi\in \Irr(G)$ of $\varrho$, as explained in the \emph{``Computational Remark 2.9''} of \cite{FG16}, where a proof of the Chevalley-Weil formula for $m=1$ is included too.
\end{remark}

\section{Pluricanonical Systems of an Abelian Cover}\label{secA}

Let \( X \cong (C_1 \times \cdots \times C_n)/G \) be a variety isogenous to a product of unmixed type. Consider the associated cover (see diagram~\eqref{VIP-diagram})
\[
\pi \colon X \to C_1' \times \cdots \times C_n',
\]
whose Galois group is \( N_{G_1 \times \cdots \times G_n}(G)/G \).

In the preceding section, we described  how to determine the character $\chi_m$ of the representation
\[
\rho_m \colon N_{G_1 \times \cdots \times G_n}(G)/G \to \mathrm{GL}(H^0(X, K_X^\tm)), \qquad g \mapsto \left( \omega \mapsto (g^{-1})^\ast \omega \right).
\]
Proposition~\ref{G-normal} asserts that the cover \( \pi \) is Galois if and only if the group \( G \) is abelian. This criterion serves as a motivation for restricting our attention to the abelian case.
In this section, we therefore consider arbitrary abelian covers \( \pi \colon X \to Y \) between complex manifolds. Our aim is to determine the pluricanonical systems \( H^0(X, K_X^\tm) \) explicitly in terms of the building data of the cover \( \pi \), i.e., by using the theory of abelian covers developed by Pardini~\cite{Par91}. While the cases \( m = 1 \) resp.~$m =2$ were previously established by Pardini \cite{Par91} and Liedtke~\cite{Lie03} resp.~by Bauer-Pignatelli \cite{BP21}, we provide a generalization for abitrary $m$. This approach not only allows us to compute the character $\chi_m$ of \( \rho_m \), but in fact yields the full representation itself.

\begin{remark}
    When $G $ is abelian, it is also possible to define the representation $ \rho_m $ via pullback by $ g $ rather than $ g^{-1} $, namely $\rho_m(g)(\omega):=g^\ast \omega$.
    For compatibility with Pardini's convention in~\cite{Par91}, we adopt this definition whenever $G$ is abelian. On the level of characters, this change corresponds to complex conjugation.
\end{remark}

Let $\pi\colon X\to Y$ be a finite abelian cover between complex manifolds with Galois group $G$, and let $m$ be a non negative integer.
The $G$-action  on  $\pi_*(K_X^{\otimes m})$ via  pullback of
holomorphic differential pluriforms yields
the eigensheaf-decomposition
\[
\pi_*(K_X^{\otimes m})=\bigoplus_{\chi \in \Irr(G)} \pi_*(K_X^{\otimes m})^\chi\,, \]
where, for an open set $U\subseteq Y$, 
\[
\pi_\ast (K_X^{\otimes m})^\chi(U):=\{\omega \in K_X^\tm(\pi^{-1}(U)) \ \mid \ g^\ast \omega= \chi(g)\cdot \omega \quad \makebox{for all} \quad  g\in G \}.
\]
Given a character ${\chi \in \Irr(G)}$,
our goal is to provide an explicit description of  the line bundle
$\pi_*(K_X^{\otimes m})^\chi$.

\begin{remark}\label{remark-Lchi}
    Following Pardini's notation  \cite{Par91}, for $m=0$ the line bundle $ (\pi_\ast \mathcal{O}_X)^\chi $ is denoted by $L_\chi^{-1}$. In particular, we have the isomorphism $L_{\chi_{triv}}\cong \mathcal{O}_Y$.
\end{remark}

First, we deal with the case where $\pi\colon X\to Y$ is a simple cyclic cover, and then we show that one can always reduce to this case.

\begin{lemma}\label{simple-cyclic-formula-1}
   Let $\pi \colon X \to Y$ be a simple cyclic cover of degree $d$ between complex manifolds defined by the line bundle $L^{\otimes d}=\mathcal O_Y(B)$ and let  $\chi_k\in \Irr(\ZZ/d)$ be the character defined by  $\chi_k(1)=\zeta_d^k$. If $m \leq k \leq d+m-1$, then it holds
 \[
 \pi_{\ast}(K_X^{\otimes m})^{\chi_k} \cong \mathcal O_Y(mB) \otimes  L^{\otimes (-k)} \otimes K_Y^{\otimes m}.
 \]
\end{lemma}

\begin{proof}
 The cover $\pi$ is locally given by 
 \[
X_i= \lbrace (x,t)\in U_i \times \CC ~\vert ~t^d-f_i(x)=0 \rbrace \to U_i, \quad (x,t) \mapsto x. 
 \]
 Here, $f_i$ is a local equation of $B$ and we have $f_i=g_{ij}^df_j$, where $g_{ij}$ is the multiplicative $1$-cocycle which defines the line bundle $L$. Note that $B$ is smooth since $X$ is smooth.  On the open set $X_i\cap \lbrace t\neq 0 \rbrace$ the bundle $K_X^{\otimes m}$ is trivialized by $(dx_1\wedge \ldots \wedge dx_n)^{\otimes m}/t^{(d-1)m}$. This implies that 
 \[
 \pi_{\ast}(K_X^{\otimes m})^{\chi_k}(U'_i)=\bigg\lbrace a_i(x)t^{k-m} \frac{(dx_1\wedge \ldots \wedge dx_n)^{\otimes m}}{t^{(d-1)m}} ~ \bigg \vert ~ a_i \in \mathcal O_Y(U'_i) \bigg\rbrace, \qquad m\leq k \leq d+m-1,
 \]
 where $ U_i':=U_i \cap\left\{f_i\neq 0\right  \}$.
 Let $\psi_{ij}(x,t)=(x,g_{ij}(x)t)$ be the transition  map $X_j \cap \pi^{-1}(U'_{ij})\to X_i \cap \pi^{-1}(U'_{ij})$. Then it holds 
 \[
a_j(x)t^{k-m} \frac{(dx_1\wedge \ldots \wedge dx_n)^{\otimes m}}{t^{(d-1)m}}= \psi_{ij}^{\ast}\bigg(a_i(x)t^{k-m} \frac{(dx_1\wedge \ldots \wedge dx_n)^{\otimes m}}{t^{(d-1)m}} 
\bigg) = a_i(x) g_{ij}(x)^{k-dm} t^{k-m}\frac{(dx_1\wedge \ldots \wedge dx_n)^{\otimes m}}{t^{(d-1)m}}. 
 \]
 This implies $a_i(x)= g_{ij}(x)^{dm-k} a_j(x)$. In other words, the collection $\lbrace a_i(x) \rbrace_{i}$ defines a section of the line bundle $L^{\otimes (dm-k)}=\mathcal O_Y(mB) \otimes L^{\otimes (-k)}$.  It follows that the map
    \[
    \mathcal O_Y(mB) \otimes L^{\otimes (-k)} \otimes K_Y^\tm \to  \pi_{\ast}(K_X^{\otimes m})^{\chi_k}, \quad 
    s \mapsto t^{k-m}\otimes\pi^\ast s,
    \] 
    is an isomorphism of line bundles.
\end{proof}

\begin{remark}\label{remark-Lchi-simple-cyclic}
    From Remark \ref{remark-Lchi} and Lemma \ref{simple-cyclic-formula-1} with $m=0$ it follows in the simple cyclic cover case 
    \begin{equation}\label{Lchik}
         L_{\chi_k}\cong L ^{\otimes k}, \qquad 0\leq k \leq d-1.
    \end{equation}
    If $k=0$, it clearly holds $L_{\chi_0}=L_{\chi_{triv}}\cong \mathcal{O}_Y$.
    If $k\neq 0$, it holds  $\chi_k^{-1}=\chi_{d-k}$, hence 
    $L_{\chi_k^{-1}} = \mathcal{O}_Y(B)\otimes L_{\chi_k}^{-1}.$
\end{remark}

\medskip

Note that the standard range for the index $k$ of the character $\chi_k\in \Irr(\ZZ/d)$ is indeed $0\leq k \leq d-1$. However, in Lemma \ref{simple-cyclic-formula-1} this range depends on $m$.  This motivates the following lemma.

\begin{lemma}\label{simple-cyclic-formula-2}
Let   $0\leq k \leq d-1$. Then it holds:
\[
\pi_{\ast}(K_X^{\otimes m})^{\chi_k} \cong \mathcal O_Y\left( \left(m-\left\lceil\frac{m-k}{d}\right\rceil \right)B \right) \otimes L_{\chi_k}^{-1} \otimes K_Y^{\otimes m}.
\]
\end{lemma}

\begin{proof}
Note that  $l:=\left\lceil\frac{m-k}{d}\right\rceil$ is the unique non negative integer such that 
$0 \leq (k-m)+ld \leq d-1$, or equivalently,
\[
m\leq k+ld\leq d+m-1.
\]
Thus, Lemma \ref{simple-cyclic-formula-1} yields
\[
\pi_{\ast}(K_X^{\tm})^{\chi_k} = \pi_{\ast}(K_X^{\tm})^{\chi_{(k+ld)}}\cong L^{\otimes (dm-k-ld)} \otimes K_Y^{\tm} 
=L^{\otimes (dm-k-ld-d+k)} \otimes L^{\otimes(d-k)}\otimes  K_Y^{\tm},
\]
where the isomorphism $L^{\otimes (dm-k-ld)} \otimes K_Y^{\tm} \to  \pi_{\ast}(K_X^{\tm})^{\chi_{(k+ld)}}$ is given by $$s\mapsto t^{r} \otimes \pi^* s\,, \qquad r:=k+ld-m.$$
If $k\neq 0$, by \eqref{Lchik} of Remark \ref{remark-Lchi-simple-cyclic} the previous chain of equalities amounts to 
\[
\pi_{\ast}(K_X^{\tm})^{\chi_k}=
\mathcal O_Y((m-l-1)B) \otimes L_{\chi_k^{-1}} \otimes K_Y^{\tm}.  
\]
If $k=0$, then we obtain 
\[
\pi_{\ast}(K_X^{\tm})^{\chi_0} =
\mathcal O_Y((m-l)B)  \otimes K_Y^{\tm}.
\]
Then the formula follows from Remark \ref{remark-Lchi-simple-cyclic}.
\end{proof}

Let us now deal with the general case. Given a finite abelian cover $\pi\colon X \to Y $ between complex manifolds with Galois group $G$, we denote by $R$ the ramification divisor  and by $D$ the branch divisor. If $T$ is an irreducible component of $R$ and $p\in T$  a general point, we denote by $H$ the stabilizer group of $p$, which in fact does not depend on the choice of $p$, and hence it is also called the stabilizer of $T$. Then, by Cartan's lemma \cite[Lemma 1]{Cartan}, the tangent representation at $p$ is faithful and looks as follows  
\begin{equation*}
    \rho_T \colon H \to \GL(n,\CC), \quad h \mapsto \operatorname{diag}(\psi(h), 1, \dotso, 1),
\end{equation*}
where $\psi$ is a generator of  $\Irr(H)$. In particular, the group $H$ is cyclic.

Given a component $C$
of the branch locus $D$, all components of $\pi^{-1}(C)$ have the same stabilizer group
and isomorphic tangent representations. Thus, we define $D_{(H, \psi)}$ to be the sum of all components of $D$ with stabilizer group $H$
and character $\psi$, and we have the following decomposition
\begin{equation*}
   D=\sum_{(H,\psi)} D_{(H,\psi)},
\end{equation*}
where $H\leq G$ runs over the cyclic subgroups of $G$ and $\psi$ over the generators of $\Irr(H)$.

Consider now a summand $D_{(H,\psi)}$ and a point $y\in D_{(H,\psi)}$ away from the other summands $D_{(H',\psi')}$. The cover $\pi$ factors as follows
\[
\xymatrix{
X\ar[rr]^\pi \ar[dr]_g & & Y\\
& Z:=X/H \ar[ur]_p & 
}
\]
Working locally around $y$, we can assume that the cover $p$ is unramified and $g$ is a simple cyclic cover of degree $d:=\vert H \vert$ defined by  $M^{\otimes d}=\mathcal O_Y(B)$. According to Remark \ref{remark-Lchi-simple-cyclic}, it holds  
\[
g_\ast \mathcal{O}_X=\bigoplus_{k=0}^{d -1} M_{\psi^k}^{-1}\cong \bigoplus_{k=0}^{d -1} M^{\otimes (-k)}.
\]
\begin{remark}\label{rmk-push-M}
    
Since $p$ is unramified and $D_{(H,\psi)}$ is the branch locus of $\pi$, we have
\begin{equation}\label{pull-back-D(H,psi)}
    p^\ast (\mathcal{O}_Y(D_{(H,\psi)}))=\mathcal{O}_Z(B) \qquad \makebox{and} \qquad p^\ast K_Y^{\otimes m}=K_Z^{\otimes m}.
\end{equation}
Moreover, by \cite[p. 207, (4.5)]{Par91} it holds
    \begin{equation}\label{push-M}
           p_\ast M^{\otimes(-k)}=\bigoplus_{\underset{\eta_{\vert H}=\psi^k}{\eta \in \Irr(G)}} L^{-1}_\eta\,, \qquad \text{ where } \quad L^{-1}_\eta= (\pi_* \mathcal O_X)^\eta.
    \end{equation}
\end{remark}
Now let $\chi\in \Irr(G)$ be a fixed character. Then we have
\begin{equation}\label{push-1}
    (\pi_\ast K_X^{\otimes m})^\chi =(p_\ast g_\ast K_X^{\otimes m})^\chi\cong\left(\bigoplus_{k=0}^{d-1} p_\ast \left( \mathcal{O}_Z\left(\mu_k B \right) \otimes M_{\psi^k}^{-1}\otimes  K_Z^{\otimes m}\right) \right)^\chi 
\end{equation}
where $\mu_k:=m-\left\lceil\frac{m-k}{d}\right\rceil$ and the isomorphism is given by applying Lemma \ref{simple-cyclic-formula-2} to $g_\ast K_X^{\otimes m}$.

By \eqref{pull-back-D(H,psi)}  each summand in 
\eqref{push-1}  is isomorphic to 
\begin{equation*}
  \left( p_\ast \left( p^\ast \left(\mathcal{O}_Y\left(\mu_{k} D_{(H,\psi)}\right) \otimes K_Y^{\otimes m}\right) \otimes M_{\psi^{k}}^{-1}\right) \right)^\chi \cong
 \mathcal{O}_Y\left(\mu_{k} D_{(H,\psi)}\right) \otimes K_Y^{\otimes m} \otimes  \left(p_\ast M_{\psi^{k}}^{-1}\right)^\chi,
\end{equation*}
where the  isomorphism is given by projection formula.
By using  \eqref{push-M} we obtain
\begin{equation*}
(\pi_\ast K_X^{\otimes m})^\chi \cong \bigoplus_{k=0}^{d-1}\left(
  \mathcal{O}_Y\left(\mu_{k} D_{(H,\psi)}\right) \otimes K_Y^{\otimes m} \otimes  \Bigg( \bigoplus_{\underset{\eta_{\vert H}=\psi^{k}}{\eta \in \Irr(G)}} L^{-1}_\eta\Bigg)^\chi \ \right)\,.
\end{equation*}
Since there exists a unique integer $0\leq k_0\leq d-1$ such that $\chi_{\vert H}=\psi^{k_0}$,
only  one summand survives:
\[
    (\pi_\ast K_X^{\otimes m})^\chi 
     \cong
  \mathcal{O}_Y\left(\mu_{k_0} D_{(H,\psi)}\right) \otimes K_Y^{\otimes m} \otimes  L^{-1}_\chi.
\]

Following the maps, we get the isomorphism of line bundles
 \[
    \mathcal{O}_Y\left(\mu_{k_0} D_{(H,\psi)}\right) \otimes K_Y^{\otimes m} \otimes  L^{-1}_\chi
    \to \left(\pi_\ast K_X^{\otimes m}\right)^\chi, \qquad 
    s  \mapsto t_{(H,\psi)}^r \otimes \pi^\ast(s),
\]
where $\{t_{(H,\psi)}=0\}$ is a local equation for the  ramification divisor over $D_{(H,\psi)}$ and the integer $r$ is defined as 
$$r:=k_0-m+\left\lceil \frac{m-k_0}{d}\right\rceil d, $$
by the isomorphism given in the proof of Lemma \ref{simple-cyclic-formula-2}.

Summing up, we have the following theorem.

\begin{theorem}\label{prop-push-decomp}
Let $\pi \colon X \to Y$ be a finite abelian cover where $X$ is normal and $Y$ is smooth with Galois group $G$. Let $m$ be a non negative integer,  
$\chi\in \Irr(G)$ a character, $H\leq G$ a cyclic subgroup  and $\psi$ a generator of $\Irr(H)$. Then there exists a unique integer  $0\leq k(H,\psi, \chi)\leq \vert H \vert -1$ such that $\chi_{\vert H}=\psi^{k(H,\psi, \chi)}$.
Moreover, define 
\[
    r(m,H,\psi, \chi):=k(H,\psi, \chi)-m + \left\lceil \frac{m-k(H,\psi, \chi)}{\vert H \vert }\right\rceil \vert H \vert
\]
and
\[
    \mu(m,H,\psi, \chi):= m-\left\lceil \frac{m-k(H,\psi, \chi)}{\vert H \vert }\right\rceil.
\]
Then there exists an isomorphism of line bundles as follows
\begin{equation}\label{push-decomposition}
    \pi_\ast(K_X^{\otimes m})^\chi\cong  \mathcal{O}_Y\left(\sum_{(H,\psi)}  \mu(m,H,\psi, \chi) \ D_{(H,\psi)}\right)\otimes K_Y^{\otimes m} \otimes L_\chi^{-1},
\end{equation}
and the $m$-canonical system has the following decomposition
\begin{equation}\label{pluricanonical-system}
    H^0(X,K_X^{\otimes m})=\bigoplus_{\chi\in \Irr(G)} \prod_{(H,\psi)} t_{(H,\psi)}^{r(m,H,\psi, \chi)}
    \ \pi^\ast H^0\left(Y,\mathcal{O}_Y\left(\sum_{(H,\psi)}     \mu(m,H,\psi, \chi) \ D_{(H,\psi)}\right)\otimes K_Y^{\otimes m} \otimes L_\chi^{-1}\right),
\end{equation}
where  $\{t_{(H,\psi)}=0\}$ is a local equation of the ramification divisor over $D_{(H,\psi)}$.

\end{theorem}

\begin{proof}
     Since $\pi_\ast(K_X^{\otimes m}) $ is reflexiv by  \cite[Corollary 1.7]{Har80}, it is enough to prove the result when $X$ is smooth. Our discussion above  shows that the line bundle isomorphism in \eqref{push-decomposition} holds in codimension 1. Hence, it holds globally by Hartog's extension theorem and the statement follows immediately.
\end{proof}

\begin{remark}\label{LBP}
Formula \eqref{push-decomposition} in
Theorem \ref{prop-push-decomp} has been derived by Pardini \cite[Proposition 4.1, c)]{Par91} for $m=1$ and by Bauer-Pignatelli \cite[Proposition 2.5]{BP21} for $m=2$. They write them as follows
    \begin{equation*}
        \pi_\ast(K_X)^{\chi}
        \cong K_Y \otimes L_{\chi^{-1}} \qquad \makebox{and} \qquad 
        \pi_\ast(K_X^{\otimes 2})^\chi\cong  \mathcal{O}_Y\bigg(\sum_{\underset{\chi_{|H}\neq\psi}{(H,\psi)}} \ D_{(H,\psi)} \bigg)\otimes K_Y^{\otimes 2} \otimes L_{\chi^{-1}}.
    \end{equation*}

In  \cite{Lie03} Liedtke derived formula \eqref{pluricanonical-system} in Theorem \ref{prop-push-decomp} for $m=1$ and wrote it as follows: 
\begin{equation*}
        H^0(X,K_X)= \bigoplus_{\chi\in \Irr(G)} \prod_{(H,\psi)} t_{(H,\psi)}^{ \vert H \vert -1 - k(1,H,\psi, \chi)}  \ \pi^\ast H^0\left(Y, K_Y \otimes L_\chi\right).
\end{equation*}
\end{remark}

For convenience of the reader, we show that the formulae from Theorem \ref{prop-push-decomp}  indeed specialize to the versions given in Remark \ref{LBP}.

Given a character $\chi\in \Irr(G)$, by \cite[(2.3) of Theorem 2.1]{Par91} it holds 
\begin{equation}\label{relation-L(chi-1)-(Lchi)-1}
    L_{\chi^{-1}}=L_{\chi}^{-1} \otimes \mathcal{O}_Y\bigg(\sum_{\underset{\chi_{|H}\neq\psi^0}{(H,\psi)}} D_{(H,\psi)}\bigg) .
\end{equation}
Hence, by using \eqref{relation-L(chi-1)-(Lchi)-1} together with \eqref{push-decomposition}, we obtain 
\begin{equation*}
    \pi_\ast(K_X^{\otimes m})^\chi \cong  \mathcal{O}_Y\left(\sum_{\underset{\chi_{|H}\neq\psi^0}{(H,\psi)}}  (\mu(m,H,\psi, \chi)-1) \ D_{(H,\psi)} + \sum_{\underset{\chi_{|H}=\psi^0}{(H,\psi)}} \mu(m,H,\psi, \chi) \ D_{(H,\psi)}\right)\otimes K_Y^{\otimes m} \otimes L_{\chi^{-1}}.
\end{equation*}

If $m=1$, then 
    \[
    \mu(1,H,\psi, \chi)=
    \begin{cases}
        0 \quad &\makebox{if} \quad k(H,\psi, \chi)=0\\
        1 \quad &\makebox{if} \quad k(H,\psi, \chi)\neq 0\\
    \end{cases}
    \]
Therefore, Pardini's formula follows. If $m=2$, then 
    \[
    \mu(2,H,\psi, \chi)=
    \begin{cases}
        1 \quad &\makebox{if} \quad k(H,\psi, \chi)=0,1\\
        2 \quad &\makebox{if} \quad k(H,\psi, \chi)\geq 2 \\
    \end{cases}
    \]
whence we obtain Bauer-Pignatelli's result. Assume now $m=1$. By using  \eqref{relation-L(chi-1)-(Lchi)-1} together with \eqref{pluricanonical-system}, where we replace $\chi^{-1}$ with $\chi$, we get
\begin{equation*}
     H^0(X,K_X)=\bigoplus_{\chi\in \Irr(G)} \prod_{(H,\psi)} t_{(H,\psi)}^{r(1,H,\psi, \chi^{-1})}
    \ \pi^\ast H^0\left(Y, K_Y \otimes L_\chi\right).
\end{equation*}
Finally, an easy computation shows that 
\[
     r(1,H,\psi, \chi^{-1})= \vert H \vert -1 - k(H,\psi, \chi),
\]
whence Liedtke's formula follows.

\subsection{Pluricanonical Systems of an Unmixed VIP with Abelian Group} \label{Pluri-Sys-Subsec}

Let $X \cong (C_1 \times \dots \times C_n)/G$ be a variety isogenous to a product of unmixed type with abelian group $G$. Recall that we have the diagram \eqref{VIP-diagram}:

\[
\xymatrix{
C_1 \times \dots \times C_n \ar[rr]^{(f_1, \dots, f_n)} \ar[dr]_\Xi &  &  C_1' \times \dots \times C_n'\\
 & X \ar[ur]_\pi & \\
} 
\]

Here, the horizontal map is the product of the quotient maps  
\[
f_j \colon C_j \to C_j' := C_j/G_j,
\]  
branched at 
\[
\mathcal{B}_j := \{p_{1,j}, \dots, p_{r_j,j}\} \subset C_j'.
\]  

In this section, we discuss to what extent one can extract information from a given algebraic datum  
\[
(G, K_1, \dots, K_n, V_1, \dots, V_n)
\]  
in order to apply Theorem \ref{prop-push-decomp} to the abelian cover $\pi$ with Galois group $(G_1 \times \dots \times G_n)/G$. More precisely, we need to determine the divisors $D_{(H,\psi)}$, the integers $k(H,\psi,\chi)$ and the line bundles $L_\chi$.

\begin{remark}\label{canonical-generator}
Consider the cover $f_j$ with generating vector  
\[
V_j := (d_{1,j}, e_{1,j}, \dots, d_{g_j',j}, e_{g_j',j}; h_{1,j}, \dots , h_{r_j,j}).
\]  
The stabilizer of a point in $C_j$ lying over $p_{i,j}$ is generated by $h_{i,j}$. Moreover, the local action of $h_{i,j}$ is given by multiplication with $\zeta_{n_{i,j}}$, where $n_{i,j} := \operatorname{ord}(h_{i,j})$.
\end{remark}

Since the cover $\Xi \colon C_1 \times \dots \times C_n \to X$ is unramified, the branch divisor $D$ of $\pi$ coincides with that of $(f_1, \dots, f_n)$. It is the union of the irreducible divisors
\[
B_{ij} := C_1' \times \dots \times C_{j-1}' \times \{p_{i,j}\} \times C_{j+1}' \times \dots \times C_n'.
\]  
On the other hand, $D$ admits the decomposition
\begin{equation}\label{branch-decomposition}
D = \sum_{(H,\psi)} D_{(H,\psi)}.
\end{equation}

We now explain how to determine the pair $(H,\psi)$ such that $B_{ij} \subset D_{(H,\psi)}$. Take a point $p = (p_1, \dots, p_n) \in C_1 \times \dots \times C_n$ mapping to a general point of $B_{ij}$, i.e., such that $f_k(p_k) \notin \mathcal{B}_k$ for $k \neq j$. Then the stabilizer of $p$ is the cyclic group
\[
(G_1 \times \dots \times G_n)_p = \langle (1, \dots, 1, h_{i,j}, 1, \dots, 1) \rangle,
\]  
whose generator acts locally around $p$ as 
\[
\operatorname{diag}(1, \dots, 1, \zeta_{n_{i,j}}, 1, \dots, 1).
\]

Next, consider the image $[p]$ in $X$. Since the cover $\Xi \colon C_1 \times \dots \times C_n \to X$ is unramified, the quotient map induces an isomorphism between the stabilizer groups:
\[
(G_1 \times \dots \times G_n)_p \overset{\sim}{\longrightarrow} ((G_1 \times \dots \times G_n)/G)_{[p]} =: H.
\]  
Under this isomorphism, the tangent representations at $p$ and $[p]$ coincide. The corresponding character $\psi \in \Irr(H)$ is then determined by
\[
\psi([(1, \dots, 1, h_{i,j}, 1, \dots, 1)]_G) = \zeta_{n_{i,j}}.
\]  

In summary, this provides a method to compute each summand in the decomposition \eqref{branch-decomposition}.

Given a character $\chi \in \Irr\big((G_1 \times \dots \times G_n)/G\big)$, we can compute $k(H, \psi, \chi)$ via
\[
\chi\big([(1, \dots, 1, h_{i,j}, 1, \dots, 1)]_G\big) = \zeta_{n_{i,j}}^{k(H, \psi, \chi)}.
\]  
Consequently, the integers $r(m, H, \psi, \chi)$ and $\mu(m, H, \psi, \chi)$ are determined immediately. The only remaining data needed to apply Theorem \ref{prop-push-decomp} are the line bundles $L_\chi$, for which the following formula holds (see \cite[(2.14) of Proposition 2.1]{Par91}):
\[
L_\chi^{\otimes d_\chi} = \mathcal{O}_Y\Bigg(\sum_{(H,\psi)} \frac{d_\chi \, k(H, \psi, \chi)}{|H|} D_{(H,\psi)} \Bigg),
\]  
where $d_\chi := \operatorname{ord}(\chi)$. Thus, we have a concrete procedure to compute $L_\chi^{\otimes d_\chi}$, which determines the line bundle $L_\chi$ up to torsion. In the special case where all curves $C_i' = \mathbb{P}^1$ we can immediately recover $L_\chi$ itself.

\section{Pluricanonical Maps of an Abelian Cover} \label{section-with-diagram}

Our goal is to study the birationality of the $m$-canonical map ($m\geq 1$) of an unmixed VIP with abelian Galois group. 
For this purpose, we consider an arbitrary abelian cover $\pi\colon X \to Y$ and recall two birationality criteria established by Catanese in \cite{Cat18}  for the canonical map of $X$. 
We present them here in the generalized setting of
the $m$-canonical map
$$\Phi_m \colon X \dashrightarrow \PP(H^0(X, K_X^\tm)^\vee).$$
The first one is a necessary criterion, which uses the character $\chi_m$ of the $m$-canonical representation 
\[
\rho_{m}\colon G \to \GL(H^0(X,K_X^\tm)), \qquad g \mapsto \ \left \{ \omega \mapsto g^\ast \omega \right \}.
\]
The second one is a sufficient criterion which additionally requires the building data of the abelian cover $\pi$. When $Y=\PP^1\times \dotso \times  \PP^1$, the second criterion becomes particularly easy to apply, especially for regular unmixed VIPs with abelian Galois group.

\medskip

Let $\pi \colon X \to Y$ be a finite abelian cover between complex manifolds with Galois group $G$. 
The group algebra 
$$\CC[G]:= \{ \sum_{g\in G} \lambda_g g \mid   \lambda_g \in \CC  \}$$
provides the \emph{regular representation of $G$} as follows
\[
\rho_{\mathrm{reg}} \colon G \to \GL(\CC[G]), \qquad h\mapsto \ \big\{ \sum_{g\in G} \lambda_g g \ \mapsto  \ \sum_{g\in G} \lambda_g (hg) \big \}.
\]
Given a general point $y \in Y$, we will construct a certain homomorphism between the representations $\rho_m$ and $\rho_{\mathrm{reg}}$. It turns out that its dual map is closely related to the restriction of the $m$-canonical map $\Phi_m$ of $X$ to the fibre $\pi^{-1}(y)$.

\medskip

Let $y\in Y $ be a point outside the branch locus of $\pi$ and fix $x_0\in \pi^{-1}(y)$. The map
\begin{equation}\label{identification-of-the-fibre}
    G\to \pi^{-1}(y), \qquad 
g \mapsto g(x_0)
\end{equation}
is a  bijection.
Furthermore, we can fix an open neighborhood $y\in W_y$ in  such a way  that there exists an open neighborhood $x_0\in U_{x_0}$ such that
\[
\pi^{-1}(W_y)=\bigsqcup_{g\in G} g(U_{x_0})
\]
and $\pi_{|U_{x_0}} \colon U_{x_0} \to W_y$ is biholomorphic. 
We may assume that  $U_{x_0}$ is the domain of a centered chart 
$$\varphi \colon U_{x_0} \to B_\sigma(0)\subset \CC^n.$$ 
Then for all $g\in G$ the composition 
\[
\varphi\circ g^{-1} \colon g(U_{x_0}) \to U_{x_0} \to  B_\sigma(0)
\]
is a centered chart around $g(x_0)$.
For every  global section $\omega  \in H^0(X,K_X^\tm)$ restricted to $g(U_{x_0})$ there is a unique $f_{g,\omega} \in \mathcal O_X(B_{\sigma}(0))$ such that 
$$(g\circ \varphi^{-1})^\ast \omega= f_{g,\omega} (dz_1 \wedge \dotso \wedge dz_n)^{\otimes m}.$$
Thus, we define the \emph{evaluation of $\omega$ at $g(x_0)$} as
$f_{g,\omega}(0)$ and denote it as follows
$$\ev_{g(x_0)}(\omega):=[(g\circ \varphi^{-1})^{\ast} \omega]_{z=0}.$$

\begin{lemma}\label{lemma-compatibility-of-the-action}
    For all $g, h\in G$ and $\omega \in H^0(X,K_X^\tm)$ it holds 
    $$\ev_{(h^{-1}g)(x_0)}(h^\ast \omega)=\ev_{g(x_0)}(\omega).$$

\end{lemma}

\begin{proof}
The claim follows directly from the definition of the evaluation by the following computation: 
\[
\ev_{(h^{-1}g)(x_0)}(h^\ast \omega)=[(h^{-1}\circ g \circ \varphi^{-1})^{\ast}h^\ast \omega ]_{z=0}=
[(g\circ \varphi^{-1})^{\ast} \omega]_{z=0}=\ev_{g(x_0)}(\omega).
\]
\end{proof}

\begin{proposition}
    The map
    \[
    \epsilon_y \colon 
    H^0(X,K_X^\tm)  \to \CC[G], \qquad 
    \omega \mapsto \sum_{g\in G} \ev_{g(x_0)}(\omega) \  \cdot g^{-1}.
    \]
    is a homomorphism of representations, i.e., for all $  h\in G$ it holds $ \epsilon_y\circ\rho_{m}(h)=\rho_{\mathrm{reg}}(h) \circ  \epsilon_y$. 
\end{proposition}

\begin{proof}
For all $h\in G$ and $\omega \in H^0(X,K_X^\tm)$, we have
    \begin{equation*}
        \epsilon_y(\rho_{m}(h)(\omega))= \epsilon_y(h^\ast \omega)=\sum_{g\in G} \ev_{g(x_0)}(h^\ast \omega ) \cdot g^{-1}.
    \end{equation*}
    Applying first the index shift given by multiplication with $h^{-1}$ and then Lemma \ref{lemma-compatibility-of-the-action} we get 
    \[
    \sum_{g\in G} \ev_{(h^{-1}g)(x_0)}(h^\ast \omega ) \cdot (h^{-1}g)^{-1} =  \sum_{g\in G} \ev_{g(x_0)}( \omega ) \cdot hg^{-1}=
     \rho_{\mathrm{reg}}(h)( \epsilon_y(\omega)).
    \]

\end{proof}

\begin{remark}
The group $G$ embeds naturally into the dual of the group algebra $\CC[G]$ via
\[
    G\to \CC[G]^\vee, \quad g \mapsto \bigg\{ \sum_{h\in G}\lambda_h h \mapsto \lambda_g \bigg\}.
\]
Using the identification $G \cong \pi^{-1}(y)$ from \eqref{identification-of-the-fibre}, we obtain the following commutative diagram:
\begin{equation*}
\xymatrixcolsep{5pc}
\xymatrix{
    G \ar[r] \ar[dd]_\wr & \CC[G]^\vee \ar[d]^-{\epsilon_y^\vee} \\
        & H^0(X, K_X^{\otimes m})^\vee \ar@{-->}[d] \\
    \pi^{-1}(y) \subset X \ar@{-->}[r]^-{\Phi_m} & \PP\big(H^0(X, K_X^{\otimes m})^\vee\big)
}
\end{equation*}
where $\epsilon_y^\vee$ denotes the dual map of $\epsilon_y$. In other words, the induced rational map
\[
\PP(\epsilon_y^\vee)\colon 
G \dashrightarrow \PP\big(H^0(X, K_X^{\otimes m})^\vee\big), 
\qquad
g \longmapsto [\ev_{g(x_0)}],
\]
coincides with the restriction of the $m$-canonical map $\Phi_m$ of $X$ to the fibre $\pi^{-1}(y)\cong G$.
In particular, if $\Phi_m$ is birational, then for general $y \in Y$ the map $\PP(\epsilon_y^\vee)$ is everywhere defined and injective.
\end{remark}

We now explain how the character $\chi_m$ of the $m$-canonical representation $\rho_m$ encodes whether $\PP(\epsilon_y^\vee)$ is injective for general $y \in Y$. This yields a purely character-theoretic necessary condition for the birationality of the $m$-canonical map $\Phi_m$. For this purpose we need to make more precise what we mean by $y\in Y$ being general. By Theorem \ref{prop-push-decomp} we have the following decomposition
\[
    H^0(X,K_X^{\otimes m})=\bigoplus_{\chi\in \Irr(G)} \prod_{(H,\psi)} t_{(H,\psi)}^{r(m,H,\psi, \chi)}
    \ \pi^\ast H^0\left(Y,  \pi_\ast\big(K_Y^{\otimes m}\big)^\chi \right),
\]
Hence, given  $\chi\in \Irr(G)$, it clearly holds that $\chi$ is a constituent of the character $\chi_m$ of $\rho_m$ if and only if $H^0(Y, (\pi_\ast K_X^{\otimes m})^{\chi})\neq 0$. Consider now the open set $U\subset Y$ defined as the complement of the union of the branch locus  of $\pi\colon X \to Y$ and the base loci of the line bundles $(\pi_\ast K_X^{\otimes m})^{\chi} $ for all constituents $\chi \in \Irr(G)$ of $\chi_m$. For each such character $\chi$ and for each point $x\in \pi^{-1}(U)$,  there exists a section $s \in H^0(Y, (\pi_\ast K_X^{\otimes m})^{\chi})$ such that
\[
\bigg( \prod_{(H,\psi)} t_{(H,\psi)}^{r(m,H,\psi, \chi)} \pi^\ast s\bigg)(x)\neq 0.
\]
This implies that for every  $y\in U$ it holds
\[
\epsilon_y( H^0(X,K_X^{\otimes m})^\chi)\neq 0.
\]
We have just shown that if $\chi\in \Irr(G)$ is a constituent of $\chi_m$ , then it is also a constituent of the character of $\epsilon_y(H^0(X, K_X^\tm)$ as a subrepresentation of $\CC[G]$. By Schur's lemma, the converse holds true as well.
\begin{remark}\label{reg-decomp}
    Since $G$ is abelian, for each $\chi\in \Irr(G)$ the isotypical component $\CC[G]^{\chi}$ is one-dimensional and generated by  the element $\sum_{g\in G} \chi(g^{-1})g$.
\end{remark}

Now we are ready to present the following proposition.

\begin{proposition}[cf.~{\cite[Corollary 2.4]{Cat18}}]\label{proj-embedding-of-G}
    Let $y\in Y$ be general. Then the map $$\PP(\epsilon_y^\vee)\colon 
G \to \PP\big(H^0(X, K_X^{\otimes m})^\vee\big)$$
  is injective if and only if for every  $h \in G\setminus \{1_G\} $ there exist irreducible constituents $\chi, \chi'$ of the character $\chi_m$ of the $m$-canonical representation $\rho_m$  such that $\chi(h)\neq \chi'(h)$.
\end{proposition}

\begin{proof} 
    Let $y\in U\subset Y$ be a general point as above and denote by $V_m$ the image $\epsilon_y(H^0(X, K_X^\tm))\subset \CC[G]$. We observe that we have a diagram as follows
    \[
    G \subset \CC[G]^\vee \twoheadrightarrow V_m^\vee  \hookrightarrow H^0(X, K_X^\tm)^\vee
    \]
    and the map $\PP(\epsilon_y^\vee)$ factors as
    \begin{equation*}
         G  \to\PP(V_m^\vee) \hookrightarrow  \PP(H^0(X, K_X^\tm)^\vee)
    \end{equation*}
    Thus, instead of  $\epsilon_y^\vee $  we can consider the map $\CC[G]^\vee \twoheadrightarrow  V_m^\vee$, which we still call $\epsilon_y^\vee $ by abuse of  notation, and the projectivization
    \[
    \PP(\epsilon_y^\vee) \colon G \to \PP(V_m^\vee),  \qquad 
         h \mapsto \big[\sum_{g\in G} \lambda_g g \mapsto \lambda_h\big].
    \]
    By Remark \ref{reg-decomp} the subrepresentation $V_m\subset\CC[G]$ is generated by the elements $\sum_{g\in G} \chi(g^{-1})g $, where  $\chi$ runs over the irreducible constituents of the character of $V_m$  which coincide with those of $\chi_m$ since $y\in U$. Let $h,k\in G$ be arbitrary elements. Hence, we have that $\PP(\epsilon_y^\vee)(h)=\PP(\epsilon_y^\vee)(k)$ if and only if there exists $\lambda \in \mathbb C^{\ast}$ such that $\epsilon_y^{\vee}(h)=\lambda \epsilon_y^{\vee}(k)$, i.e.,
\[
\chi(h^{-1})= \lambda \chi(k^{-1}) 
\]
for every irreducible constituent $\chi$ of $\chi_m$. In other words, $\PP(\epsilon_y^\vee)(h)=\PP(\epsilon_y^\vee)(k)$ if and only if $\chi(h^{-1}k)=\chi'(h^{-1}k)$ for all $\chi, \chi'$ irreducible constituents of $\chi_m$. 
Since $h, k \in G$ are arbitrary, we are done.
\end{proof}

As a corollary, we obtain the announced character-theoretic  necessary criterion for the birationality of $\Phi_m$.

\begin{corollary}\label{corollary-necessary-bir-criterion}
    Let $\pi \colon X \to Y$ be a finite abelian cover between complex manifolds with Galois group $G$. Assume that the $m$-canonical map $\Phi_m$ of $X$ is birational. Then for every  $h \in G\setminus \{1_G\} $ there exist irreducible constituents $\chi, \chi'$ of the character $\chi_m$ of the $m$-canonical representation $\rho_{m}$ of $\pi$ such that $\chi(h)\neq \chi'(h)$.
\end{corollary}

Next we present  the sufficient  criterion for the birationality of $\Phi_m$ announced at the beginning of this section.

\begin{proposition}[cf.~{\cite[Proposition 3.1]{Cat18}}]\label{birationality-criteria}

    Let $\pi \colon X \to Y$ be a finite abelian cover between complex manifolds with Galois group $G$ such that the following conditions hold:
    \begin{itemize}
        \item[$(a)$]  $\makebox{For all} \ h\in G \setminus\{1_G\}$ there exist irreducible constituents $\chi, \chi'$ of $\chi_m$  such that $\chi(h)\neq \chi'(h)$.
        
        \item[$(b)$] 

        For every pair of general points $y,y'\in Y$ there exists an irreducible constituent $\chi$ of $\chi_m$  such that $H^0(Y, (\pi_\ast K_X^{\otimes m})^{\chi})$ separates $y,y'$.
    \end{itemize}
    Then  the $m$-canonical map $\Phi_{m}$ of $X$ is birational.
\end{proposition}

\begin{proof}
    By Proposition \ref{proj-embedding-of-G} $(a)$ means that $\Phi_m$ is injective on a general fibre $\pi^{-1}(y)$. Given $y,y'\in Y$ general,
    condition $(b)$ implies that
    $x\in \pi^{-1}(y)$ and $x'\in \pi^{-1}(y')$ are separeted by the sublinear system 
     \[
     \prod_{(H,\psi)} t_{(H,\psi)}^{r(m,H,\psi, \chi)}
     \ \pi^\ast H^0\left(Y,  \pi_\ast\big(K_Y^{\otimes m}\big)^\chi \right) \subset H^0(X,K_X^\tm).
     \]
\end{proof}

When  $Y=\PP^1 \times \dotso \times \PP^1$, we have 
\[
(\pi_\ast K_X^{\otimes m})^{\chi}=\mathcal{O}_Y(r^1_{m,\chi}, \dotso, r^n_{m,\chi}), 
\]
where $r^j_{m,\chi}\in \ZZ$ for $j=1, \dotso, n$.
Hence, $\chi$ is an irreducible constituent of $\chi_m$ if and only if $r^j_{m,\chi}\geq 0$ for all $j$.

\begin{corollary}[cf.~{\cite[Corollary 3.2]{Cat18}}]\label{birational-criteria-Y=P1^n}
Let $\pi \colon X \to (\PP^1)^n$ be a finite abelian cover with $X$ smooth and Galois group $G$. Assume that the following conditions hold:
    \begin{itemize}
        \item[$(a)$]  $\makebox{For all} \ h\in G \setminus\{1_G\}$ there exist irreducible constituents $\chi, \chi'$ of $\chi_m$  such that $\chi(h)\neq \chi'(h)$.

        \item[$(b)$]

        For every $j=1\dotso, n$ there exists an irreducible constituent $\chi$  of $\chi_m$ such that $r_{m,\chi}^j \geq 1$.
    \end{itemize}
    Then  the $m$-canonical map $\Phi_{m}$ of $X$ is birational.
\end{corollary}

\begin{proof}
   Condition $(a)$ corresponds exactly to condition $(a)$ of Proposition \ref{birationality-criteria}. So it remains to show that condition $(b)$ implies condition $(b)$ of Proposition \ref{birationality-criteria}. Given integers $r_j\geq 0$ for $j=1,\dotso, n$, the map associated with the line bundle $\mathcal{O}_{(\PP^1)^n}(r_1, \dotso, r_n)$ factors as 
   \[
   \PP^1 \times \dotso \times \PP^1 \to \PP^{N_1} \times \dotso \times \PP^{N_n} \hookrightarrow \PP^N,
   \]
   where the first is the product of the maps associated with $\mathcal{O}_{\PP^1}(r_j) $ and the second is the Segre embedding. Hence, the linear system
   $$H^0\big((\PP^1)^n, \mathcal{O}_{(\PP^1)^n}(r_1, \dotso, r_n)\big) $$
   separates two distinct points $y=(y_1, \dotso, y_n)$ and $y'=(y_1', \dotso, y_n')$  if and only if there exists $j$ such that $y_j\neq y_j'$ and $r_j\geq 1$. Hence, if  for every $j$ there is an irreducible constituent $\chi$ of $\chi_m$ such that $r_{m,\chi}^j\geq 1$, then every two points $y,y'\in Y$ are separated, and this clearly implies condition $(b)$ of Proposition \ref{birationality-criteria}.
\end{proof}

\section{Examples and Computational Results} \label{FinalSection}

\subsection{Example of a Threefold with Birational Canonical Map}

In \cite{FG16}, the authors provide an algorithm to compute all algebraic data of unmixed VIPs with trivial kernels for a fixed value of the holomorphic Euler--Poincaré characteristic~$\chi$. We adapt this algorithm to the case of regular unmixed VIPs with trivial kernels and abelian group in order to determine whether the variety corresponding to a given algebraic datum in the output admits a birational $m$-canonical map.  
To this end, we implement the method described in Section~\ref{Pluri-Reps-Sec}, i.e., the Chevalley-Weil formula, to compute the character $\chi_m$ of the $m$-canonical representation from the algebraic datum. Since these VIPs are regular, we can also compute all the necessary data to apply the formulae in Theorem~\ref{prop-push-decomp}, as explained in Subsection~\ref{Pluri-Sys-Subsec}. This allows us to test whether the birationality criterion from Corollary~\ref{birational-criteria-Y=P1^n} is satisfied and, in addition, whether the $m$-canonical system is basepoint free.  

Running the algorithm for threefolds with $m=1$ and $-8 \leq \chi \leq -1$, we find out that there are no threefolds with $-7 \leq \chi \leq -1$ fulfilling the above birationality criterion, while examples do occur for $\chi = -8$. Next, we present one such example in detail and also analyze the differential of its canonical map.

\smallskip

Consider the group $G:=(\ZZ/2)^3$ generated by the unitary vectors $e_1,e_2,e_3$.
The generating vectors 
\begin{equation*}
    \begin{split}
        &V_1=(e_1, \ e_1,\  e_2, \ e_2, \ e_3, \ e_3), \\
        &V_2=( e_1e_3,\ e_1e_3,\ e_1e_2, \ e_1e_2, \ e_1e_2e_3, \ e_1e_2e_3), \\ 
        &V_3=(e_1e_3, \ e_1e_3, \ e_2e_3, \ e_2e_3,\ e_1e_2e_3,\ e_1e_2e_3)
    \end{split}
\end{equation*}
of $G$ provide an algebraic datum of a regular threefold $$X:=(C_1\times C_2 \times C_3)/G$$ isogenous to a product and  with trivial kernels $K_i$. 
By running the algorithm described above, we see that the canonical map $\Phi_1\colon X \dashrightarrow \PP(H^0(X, K_X)^\vee)$ is birational and basepoint free. Hence, by Proposition \ref{bpfreeness}, it is a finite birational morphism, and there are two possibilities: either the differential of $\Phi_1$ is everywhere injective and then $\Phi_1$ is an embededding, or $\Phi_1$ is the normalization map of a singular threefold. 
In order to analyze the differential of $\Phi_1$, we derive and use explicit equations for the curves $C_i$'s.

We begin with $C_1$.
Applying the Chevalley-Weil formula \cites{CheWeil34,Weil35}, we get the character of the canonical representation $H^0(C_1, K_{C_1})$; since $G$ is abelian, we get right away the canonical representation itself:
$$\psi_1 \colon G \to \GL(5, \CC)$$
\[
e_1\mapsto \operatorname{diag}(-1,-1,1,-1,-1),
\qquad 
e_2 \mapsto  \operatorname{diag}(-1,-1,-1,1,-1),
\qquad 
e_3 \mapsto \operatorname{diag}(-1,-1,-1,-1,1).
\]
It is easy to see that the curve $C_1$ is not hyperelliptic. Suppose by contradiction the contrary. The hyperelliptic involution $\iota$ acts on $H^0(C_1, K_{C_1})$ as multiplication by $-1$, but the image of $\psi_1$ does not contain the matrix $-I_5$. This implies that $G$ embeds into $\Aut(\PP^1)=\PGL(2,\CC)$, a contradiction since the finite subgroups of $\PGL(2,\CC)$ are cyclic, dihedral, $A_4$, $S_4$ or $A_5$ (cf. \cite{Klein}). Thus, the curve $C_1$ is canonically embedded in $\PP^4$, where we choose  $(x_0: x_1: x_2: x_3: x_4)$ as homogeneous coordinates.
Consider the Castelnuovo-Noether sequence (cf.~\cite[p. 253]{PrincipleAG})
\[
0 \to \mathcal I_2 \to \Sym^2(H^0(C_1, K_{C_1})) \to H^0(C_1, K_{C_1}^{\otimes 2}) \to 0,
\]
where $\mathcal{I}_2\subset \CC[x_0, x_1, x_2, x_3, x_4]_2$ denotes the space of quadratic forms which vanish along $C_1$. 
Hence, the character  $\chi_{\mathcal{I}_2}$ of $ \mathcal I_2$ is the difference between the character of $\Sym^2(H^0(C_1, K_{C_1})) $ and that of $H^0(C_1, K_{C_1}^{\otimes 2}) $, which we can compute by the Chevalley-Weil formula for pluricanonical representations, i.e., Theorem \ref{Chev-Weil-Pluri}. It turns out that $\chi_{\mathcal{I}_2}=3\chi_{triv}$. In other words, the action of $G$ on $\mathcal I_2$ is trivial and $\dim_\CC \mathcal{I}_2=3$. The space of $G$-invariant quadratic forms is
\[
\{a_0 x_0^2 + a_1 x_1^2 +a_2 x_2^2 + a_3 x_3^2 +a_4 x_4^2+ a_{01}x_0x_1 \ \mid \ a_0,\dotso, a_4, a_{01}\in \CC\}.
\]
By Riemann's existence theorem, there is a unique three-dimensional family of $G$-covers of $\PP^1$ branched at 6 points with respect to the given action $\psi_1$. This leads us to consider the family of $G$-invariant
algebraic sets
$C_{a,b,c}\subset \PP^4$ given by
\begin{equation}\label{equations-C}
\left\{\begin{array}{l}      x_0^2+x_1^2+x_2^2+ax_0x_1=0\\
        x_0^2+x_1^2+x_3^2+bx_0x_1=0 \\
        x_0^2+x_1^2+x_4^2+cx_0x_1=0
        \end{array}\right.
\end{equation}
where $a,b,c\in \CC\setminus\{\pm 2\}$, $a\neq b$, $a\neq c$, $b\neq c$. Indeed, the algebraic set $C_{a,b,c}\subset \PP^4$ is a canonically embedded smooth projective curve of genus 5.

Furthermore,  the projection map
\[
 C_{a,b,c} \to \PP^1, \quad (x_0\colon x_1 \colon x_2 \colon x_3 \colon x_4) \mapsto (x_0\colon x_1)
\]
is $G$-equivariant and has degree 8, hence it is Galois. It is branched at 6 points, more precisely:

\begin{itemize}
    \item  $e_1$ fixes 8 points on $C_{a,b,c} \cap \{x_2=0\}$;  the corresponding branch points are 
    \[
    \left(\frac{-a+\sqrt{a^2-4}}{2}\colon 1\right), \quad  \left(\frac{-a-\sqrt{a^2-4}}{2}\colon 1\right).
    \]
    \item $e_2$ fixes 8 points on $C_{a,b,c} \cap  \{x_3=0\}$; the corresponding branch points are 
    \[
    \left(\frac{-b+\sqrt{b^2-4}}{2}\colon 1\right), \quad  \left(\frac{-b-\sqrt{b^2-4}}{2}\colon 1\right).
    \]
    \item $e_3$ fixes 8 points on $C_{a,b,c} \cap  \{x_4=0\}$;
    the corresponding branch points are 
    \[
    \left(\frac{-c+\sqrt{c^2-4}}{2}\colon 1\right), \quad  \left(\frac{-c-\sqrt{c^2-4}}{2}\colon 1\right).
    \]
\end{itemize}
 
Hence, the curve $C_1$ belongs to the family $\{C_{a,b,c}\}$. Applying to $V_1$ the automorphim of $G$  given by 
\[
A_2:=
\begin{pmatrix}
    1 & 1& 1 \\
       0&  1   &1 \\
      1&  0 & 1 \\
\end{pmatrix} \qquad \makebox{resp.} \qquad
A_3:=
\begin{pmatrix}
    1 & 0& 1 \\
    0&  1   &1 \\
      1&  1 & 1 \\
\end{pmatrix},
\]
we obtain $V_2$ resp. $V_3$. Therefore, the curves $C_2$ and $C_3$ belong to the same family described above.
The twisted actions are
\[
\psi_2:=\psi_1 \circ A_2^{-1}
\quad 
\makebox{and}
\quad 
\psi_3:=\psi_1 \circ A_3^{-1}.
\]

From what we have said, it follows that 
\[
C_1 \times C_2 \times C_3 \hookrightarrow \PP^4_{\mathbf{x}} \times \PP^4_{\mathbf{y}} \times \PP^4_{\mathbf{z}}
\]
and the canonical map of $X$ is given by the $G$-invariant monomials $x_iy_jz_k$, i.e.,
\[
\Phi_1 \colon (C_1\times C_2 \times C_3)/G \to \PP^{15}
\]
\[
\begin{split}
    [( \mathbf{x},\mathbf{y},\mathbf{z})] \mapsto ( &x_0y_3z_4:  \  x_0y_4z_3: \ x_1y_3z_4:\  x_1y_4z_3:\ x_2y_2z_0: \ x_2y_2z_1: \ x_2y_3z_2: \  x_3y_0z_0 \\
        & x_3y_0z_1:  \  x_3y_1z_0: \ x_3y_1z_1:\  x_3y_3z_3: \ x_3y_4z_4: \ x_4y_0z_2: \ x_4y_1z_2: \ x_4y_2z_3). 
\end{split}
\]
We already know that $\Phi_1$ is a finite birational morphism. Next, we analyze the injectivity of its differential. To this end, we can compose $\Phi_1$ with the quotient map 
$$ C_1\times C_2 \times C_3 \to (C_1\times C_2 \times C_3)/G.$$
Since this quotient map is unramified, the differential $dF$ of the composition $F$ is injective if and only if  $d\Phi_1$ is injective. 
We are going to show that there exists a point where the differential $dF$ is not injective. For this purpose, 
we work on the affine open set
$x_3y_3z_3\neq 0$: 

\begin{equation*}
    U_{1}\colon \
    \begin{cases}
        x_0^2+x_1^2+ax_0x_1+x_2^2=0\\
        x_0^2+x_1^2+cx_0x_1+x_4^2=0\\
        x_0^2+x_1^2+bx_0x_1+1=0 \\
    \end{cases},
   \quad
   U_{2} \colon \
   \begin{cases}
        y_0^2+y_1^2+a'y_0y_1+y_2^2=0\\
        y_0^2+y_1^2+c'y_0y_1+y_4^2=0\\
        y_0^2+y_1^2+b'y_0y_1+1=0 \\
    \end{cases}, \quad 
   U_{3} \colon \
   \begin{cases}
        z_0^2+z_1^2+a''z_0z_1+z_2^2=0\\
        z_0^2+z_1^2+c''z_0z_1+z_4^2=0\\
        z_0^2+z_1^2+b''z_0z_1+1=0 \\
    \end{cases}.
\end{equation*}

\medskip
Then, locally on this affine set,  the map $F$ is given by 
\[
\begin{split}
   F\colon   U_{1} \times  U_{2} \times    U_{3}  \to & \CC^{15}\\(x_0,x_1,x_2,x_4,y_0,y_1,y_2,y_4, z_0, z_1,z_2,z_4)\mapsto ( &x_0z_4,  \  x_0y_4, \ x_1z_4,\  x_1y_4,\ x_2y_2z_0, \\
    & x_2y_2z_1, \ x_2z_2, \ y_0z_0, \ y_0z_1,  \  y_1z_0, \\
    & y_1z_1,\   y_4z_4, \ x_4y_0z_2, \ x_4y_1z_2, \ x_4y_2).
\end{split}
\]
\newpage

\begin{remark}\label{local-parameters}
The Jacobian of the affine curve $U_{1}$ is 
\[
    \begin{pmatrix}
        2x_0+ax_1 & 2x_1+ax_0 & 2x_2  & 0\\
        2x_0+cx_1 & 2x_1+cx_0 & 0 & 2x_4 \\
        2x_0+bx_1 & 2x_1+bx_0 & 0 &  0 \\
    \end{pmatrix}.
\]
Since $a,b,c\in \CC\setminus\{\pm 2\}$, at each point of $U_{1}$ it holds 
\[
2x_1+\alpha x_0\neq 0 \quad \makebox{or} \quad 2x_0+\alpha x_1\neq 0,
\]
where $\alpha\in \{a,b,c\}$. 
By the implicit function theorem it follows  that, locally around a point of $ U_{1}$, one of the following conditions holds:
\begin{enumerate}[(i)]
    \item if $x_2x_4\neq 0$ and also $2x_1+bx_0\neq 0$, then we can take $x_0$  as a local parameter;

    \item if $x_2x_4\neq 0$ and also $2x_0+bx_1\neq 0$, then we can take $x_1$  as a local parameter; 

    \item if $x_2=0$, then we can take $x_2$ as a local parameter, since $4(a-b)x_4(x_0^2-x_1^2)\neq 0$;

    \item if $x_4=0$, then we can take $x_4$ as a local parameter, since $4(b-c)x_2(x_0^2-x_1^2)\neq 0$.
\end{enumerate}
In the same way we can choose local parameters for the curves $U_2$ and $U_3$.
\end{remark}

Now consider a point $p=(\mathbf{x}, \mathbf{y}, \mathbf{z})\in  U_{1} \times  U_2 \times    U_3 $ such that $x_0=x_1\neq 0$, $x_2x_4\neq 0$ and $y_4=z_4=0$. Notice that such a point exists: in fact, we can take  $p$ with coordinates 
\[
     \qquad \qquad \mathbf{x}=\left(\frac{i}{\sqrt{b+2}}, \ \frac{i}{\sqrt{b+2}}, \ \sqrt{\frac{a+2}{b+2}}, \ \sqrt{\frac{c+2}{b+2}}\right),
 \]
 
 \[
 \begin{split}
     &y_0^2=\frac{c'(b'-c')\pm \sqrt{(c'^2-4)(b'-c')^2}}{2(b'-c')^2}, \quad
     y_1=-\frac{1}{(b'-c')y_0}, \quad
     y_2^2=-y_0^2-y_1^2-a'y_0y_1, \quad y_4=0, \\
     &z_0^2=\frac{c''(b''-c'')\pm \sqrt{(c''^2-4)(b''-c'')^2}}{2(b''-c'')^2}, \quad 
     z_1=-\frac{1}{(b''-c'')z_0}, \quad z_2^2=-z_0^2-z_1^2-a''z_0z_1, \quad z_4=0.
 \end{split}
 \]
 
Hence, by Remark \ref{local-parameters} we can take $x_0, y_4,z_4$ as local parameters around $p$ and the Jacobian matrix of $F$ reads:
\begin{equation*}
J:=
\begin{pmatrix}
    z_4 &  0& x_0 \\ 
    y_4 &  x_0 & 0 \\ 
    \frac{\de x_1}{\de x_0}z_4 & 0 & x_1 \\ 
    \frac{\de x_1}{\de x_0} y_4 & x_1  & 0\\
    \frac{\de x_2}{\de x_0} y_2z_0 & x_2z_0 \frac{\de y_2}{\de y_4}  & x_2y_2 \frac{\de z_0}{\de z_4}\\
    \frac{\de x_2}{\de x_0} y_2z_1 & x_2z_1\frac{\de y_2}{\de y_4}  &  x_2y_2 \frac{\de z_1}{\de z_4} \\ 
    \frac{\de x_2}{\de x_0} z_2 & 0 & x_2 \frac{\de z_2}{\de z_4}\\ 
    0 &  z_0\frac{\de y_0}{\de y_4} & y_0 \frac{\de z_0}{\de z_4}\\
    0 & z_1 \frac{\de y_0}{\de y_4} & y_0 \frac{\de z_1}{\de z_4}\\ 
    0 & z_0\frac{\de y_1}{\de y_4} & y_1 \frac{\de z_0}{\de z_4}\\ 
    0 & z_1 \frac{\de y_1}{\de y_4}  & y_1 \frac{\de z_1}{\de z_4}\\
    0 & z_4 & y_4\\ 
    \frac{\de x_4}{\de x_0} y_0z_2 & x_4z_2 \frac{\de y_0}{\de y_4}  & x_4y_0 \frac{\de z_2}{\de z_4}\\
    \frac{\de x_4}{\de x_0} y_1z_2 & x_4z_2\frac{\de y_1}{\de y_4}  &  x_4y_1 \frac{\de z_2}{\de z_4}\\ \frac{\de x_4}{\de x_0} y_2 & x_4\frac{\de y_2}{\de y_4}& 0\\ 
\end{pmatrix}
\end{equation*}
where the derivatives with respect to $x_0$, $y_4$ and $z_4$ are given by implict differentiation as follows
\begin{equation*}
     \begin{pmatrix}
         \frac{\de x_1}{\de x_0} \\
         \frac{\de x_2}{\de x_0} \\
         \frac{\de x_4}{\de x_0} 
     \end{pmatrix}=
     -
      \begin{pmatrix}
         2x_1+ax_0 & 2x_2  & 0\\
         2x_1+cx_0 & 0 & 2x_4 \\
         2x_1+bx_0 & 0 &  0 \\
    \end{pmatrix}^{-1}
     \begin{pmatrix}
         2x_0+ax_1\\
         2x_0+cx_1\\
         2x_0+bx_1
     \end{pmatrix}=
     \begin{pmatrix}
         -(2x_0 +bx_1)/(2x_1+bx_0)\\
  (a-b)(x_0^2 -x_1^2)/(2x_1+bx_0)x_2 \\
  (c-b)(x_0^2 -x_1^2)/(2x_1+bx_0)x_4
     \end{pmatrix},
\end{equation*}

\begin{equation*}
     \begin{pmatrix}
         \frac{\de y_0}{\de y_4} \\
         \frac{\de y_1}{\de y_4} \\
         \frac{\de y_2}{\de y_4} 
     \end{pmatrix}=
     -
      \begin{pmatrix}
         2y_0+a'y_1 & 2y_1+a'y_0 & 2y_2 \\
         2y_0+c'y_1 & 2y_1+c'y_0 & 0 \\
        2y_0+b'y_1 & 2y_1+b'y_0 & 0
    \end{pmatrix}^{-1}
     \begin{pmatrix}
        0\\
         2y_4\\
         0
     \end{pmatrix}=
     \begin{pmatrix}
         -y_4(2y_1+b'y_0)/(b'-c')(y_0^2-y_1^2) \\
  y_4(2y_0+b'y_1)/(b'-c')(y_0^2-y_1^2) \\
  -y_4(a'-b')/(b'-c')y_2
     \end{pmatrix},
\end{equation*}

\begin{equation*}
     \begin{pmatrix}
         \frac{\de z_0}{\de z_4} \\
         \frac{\de z_1}{\de z_4} \\
         \frac{\de z_2}{\de z_4} 
     \end{pmatrix}=
     -
      \begin{pmatrix}
         2z_0+a''z_1 & 2z_1+a''z_0 & 2z_2 \\
         2z_0+c''z_1 & 2z_1+c''z_0 & 0 \\
        2z_0+b''z_1 & 2z_1+b''z_0 & 0
    \end{pmatrix}^{-1}
     \begin{pmatrix}
        0\\
         2z_4\\
         0
     \end{pmatrix}=
     \begin{pmatrix}
         -z_4(2z_1+b''z_0)/(b''-c'')(z_0^2-z_1^2) \\
  z_4(2z_0+b''z_1)/(b''-c'')(z_0^2-z_1^2) \\
  -z_4(a''-b'')/(b''-c'')z_2
     \end{pmatrix}.
 \end{equation*}

In particular, at the point $p$  we have that
$$\frac{\de x_2}{\de x_0}=\frac{\de x_4}{\de x_0}=\frac{\de y_j}{\de y_4}=\frac{\de z_j}{\de z_4}=0 \qquad \makebox{for} \quad  j=0,1,2,$$
and thus the first column of the Jacobian matrix $J$ is zero.

\medskip

Next, we show that such a point $p$ maps via $F$ to an isolated singularity of the canonical image of $X$. To this end, we show that in the given neighborhood of $p$ there is a finite number of points at which the differential of $F$ is not injective.  
Suppose that the Jacobian matrix $J$ has not full rank. Then all $3 \times 3$ minors vanish, in particular
\begin{itemize}
    \item the minor given by rows 1,2,3 
    \[
    \begin{pmatrix}
      z_4 &  0& x_0 \\ 
    y_4 &  x_0 & 0 \\ 
    \frac{\de x_1}{\de x_0}z_4 & 0 & x_1 \\ 
    \end{pmatrix}=x_0z_4\big(x_1-x_0\frac{\de x_1}{\de x_0}\big)=0 \quad
    \implies \quad x_0z_4=0;
    \]
    \item the minor given by rows 1,3,4    
    \[
    \begin{pmatrix}
     z_4 &  0& x_0 \\ 
    \frac{\de x_1}{\de x_0}z_4 & 0 & x_1 \\ 
    \frac{\de x_1}{\de x_0} y_4 & x_1  & 0\\
    \end{pmatrix}=-x_1\big(x_1z_4-x_0z_4 \frac{\de x_1}{\de x_0}\big) =-x_1^2z_4=0.
    \]
\end{itemize}

Since $x_0\neq 0$ or $x_1\neq 0$, this implies $z_4=0.$
We want to show that also $y_4=0$. 
Assume by contradiction $y_4\neq 0$. Since either $z_0\neq 0$ or $z_1\neq 0$, when we consider all the minors given by row 2, row $i$ with $i\in \{8, \dotso, 11\}$ and row 12, we obtain that
\[
\frac{\de y_0}{\de y_4}=\frac{\de y_1}{\de y_4}=0 \quad \implies \quad 2y_1+b'y_0=2y_0+b'y_1=0 \quad \implies y_0=y_1=0,
\]
where the last implication holds since $b'^2 \neq 4$. But $y_0=y_1=0$ is a contradiction. Thus, $y_4=0$.
Since $y_2\neq 0$, $z_2\neq 0$ and either  $x_0\neq 0$ or $x_1\neq 0$, we get that 
\[
 \frac{\de x_2}{\de x_0}= \frac{\de x_4}{\de x_0}=0 \quad \implies \quad x_0^2=x_1^2.
\]
Finally, the point in the neighborhood of $p$ we started with must fulfill
$$z_4=y_4=0=x_0^2-x_1^2.$$
We are done, since there are only finitely many points on 
$$U_1\cap\{x_0^2-x_1^2=0\} \times U_2 \cap \{y_4=0\} \times U_3\cap \{z_4=0\}.$$
In conclusion, the canonical map $\Phi_1 \colon X \to \PP^{15}$ is a finite birational morphism, but it is not an embedding. Moreover, there exists a non-normal isolated singularity on the canonical image of $X$.

Therefore, we have proven the following theorem.

\begin{theorem}\label{3-fold-normalization-map}
    There exists a regular smooth projective threefold $X$ whose canonical map $\Phi_1\colon X\to \Sigma \subset \PP^{15}$ is the normalization map. In particular, $\Sigma$ has degree $K^3_X=-48\cdot\chi(\mathcal{O}_X)=384$ and at least one isolated non-normal singularity. 
\end{theorem}

\subsection{Further Computational Results}
In this subsection we present some further computational results.
More precisely, we run a version of the classification algorithm for regular unmixed
three-dimensional VIPs with abelian groups that also works with non-trivial kernels.
The algorithm takes as input a fixed value of the holomorphic Euler–Poincaré
characteristic $\chi \leq -1$ and an upper bound on the group order.
Although the group order is bounded in terms of $\chi$, the resulting bound is too
large for a complete classification:
\[
|G| \leq 84^{6}\,\chi(\mathcal O_X)^2,
\]
cf.\ \cite[Proposition~4.4]{FG16}. 
Running our code for $\chi=-1$ and $|G| \leq 16$, we obtain $347 $ families of
threefolds listed in Table \ref{Table-chi-1}. Each row in the table contains the Galois group $G$,
the orders of the kernels $k_i$, the branching types $T_i$ of the Galois covers
$C_i \to C_i/G_i \simeq \mathbb{P}^1$ and the Hodge numbers $h^{i,j}$ of the threefolds.
Additionally, for the canonical and bicanonical maps, the table indicates
the number of families for which these maps are base-point free, birational, non-birational, and also of
those for which we cannot decide: namely, part (a) of Corollary~\ref{birational-criteria-Y=P1^n} is true, but part (b) is false.

\begin{longtable}{|c|c|c|c|c|c|c|c|c|c|c|c|c|c|c|c|c|c|c|c|c|}
\caption{The 347 families of threefolds with $\chi=-1$ and $|G|\leq 16$} \label{Table-chi-1} \\

\cline{14-21}

\multicolumn{13}{c|}{} 
& \multicolumn{4}{c|}{canonical map} 
& \multicolumn{4}{c|}{bicanonical map} \\
\cline{14-21}

\hline
No. & $G$ & $k_1$ & $k_2$ & $k_3$ 
& $T_1$ & $T_2$ & $T_3$
& $h^{3,0}$ & $h^{2,0}$ & $h^{1,0}$ 
& $h^{1,1}$ & $h^{2,1}$ 
& bpf & bir & nbir & ? 
& bpf & bir & nbir & ? \\
\hline
\endfirsthead

\cline{14-21}

\multicolumn{13}{c|}{} 
& \multicolumn{4}{c|}{canonical map} 
& \multicolumn{4}{c|}{bicanonical map} \\
\cline{14-21}

\hline
No. & $G$ & $k_1$ & $k_2$ & $k_3$ 
& $T_1$ & $T_2$ & $T_3$
& $h^{3,0}$ & $h^{2,0}$ & $h^{1,0}$ 
& $h^{1,1}$ & $h^{2,1}$ 
& bpf & bir & nbir & ? 
& bpf & bir & nbir & ?  \\
\hline
\endhead

\hline 


\endfoot

\hline
\endlastfoot

1 & $\ZZ_2^3$ & 1 & 1 & 1 & $2^5$ & $2^5$ & $2^5$ & 4 & 2 & 0 & 7 & 12 & 0 & 0 & 9 & 0 & 9 & 9 & 0 & 0 \\
2 & $\ZZ_2^3$ & 1 & 1 & 1 & $2^5$ & $2^5$ & $2^5$ & 5 & 3 & 0 & 9 & 15 & 2 & 0 & 8 & 0 & 8 & 8 & 0 & 0 \\
3 & $\ZZ_2^3$ & 1 & 1 & 2 & $2^5$ & $2^5$ & $2^6$ & 3 & 1 & 0 & 5 & 9 & 0 & 0 & 2 & 0 & 2 & 0 & 2 & 0 \\
4 & $\ZZ_2^3$ & 1 & 1 & 2 & $2^5$ & $2^5$ & $2^6$ & 4 & 2 & 0 & 7 & 12 & 0 & 0 & 2 & 0 & 2 & 2 & 0 & 0 \\
5 & $\ZZ_2^3$ & 1 & 1 & 2 & $2^5$ & $2^5$ & $2^6$ & 5 & 3 & 0 & 9 & 15 & 0 & 0 & 6 & 0 & 6 & 6 & 0 & 0 \\
6 & $\ZZ_2^3$ & 1 & 1 & 2 & $2^5$ & $2^5$ & $2^6$ & 6 & 4 & 0 & 11 & 18 & 1 & 0 & 1 & 0 & 1 & 1 & 0 & 0 \\
7 & $\ZZ_2^3$ & 1 & 1 & 2 & $2^5$ & $2^6$ & $2^5$ & 4 & 2 & 0 & 7 & 12 & 0 & 0 & 4 & 0 & 4 & 0 & 4 & 0 \\
8 & $\ZZ_2^3$ & 1 & 1 & 2 & $2^5$ & $2^6$ & $2^5$ & 5 & 3 & 0 & 9 & 15 & 0 & 0 & 2 & 0 & 2 & 0 & 2 & 0 \\
9 & $\ZZ_2^3$ & 1 & 1 & 4 & $2^5$ & $2^6$ & $2^6$ & 4 & 2 & 0 & 7 & 12 & 0 & 0 & 2 & 0 & 2 & 0 & 2 & 0 \\
10 & $\ZZ_2^3$ & 1 & 1 & 4 & $2^5$ & $2^6$ & $2^6$ & 6 & 4 & 0 & 11 & 18 & 1 & 0 & 1 & 0 & 1 & 0 & 1 & 0 \\
11 & $\ZZ_2^3$ & 1 & 2 & 2 & $2^5$ & $2^6$ & $2^6$ & 4 & 2 & 0 & 7 & 12 & 0 & 0 & 2 & 0 & 2 & 0 & 2 & 0 \\
12 & $\ZZ_2^3$ & 1 & 2 & 2 & $2^5$ & $2^6$ & $2^6$ & 5 & 3 & 0 & 9 & 15 & 0 & 0 & 2 & 0 & 2 & 0 & 2 & 0 \\
13 & $\ZZ_3^2$ & 1 & 1 & 3 & $3^4$ & $3^4$ & $3^4$ & 4 & 2 & 0 & 7 & 12 & 0 & 0 & 2 & 0 & 2 & 0 & 0 & 2 \\
14 & $\ZZ_2^4$ & 1 & 1 & 4 & $2^5$ & $2^5$ & $2^5$ & 2 & 0 & 0 & 3 & 6 & 0 & 0 & 2 & 0 & 2 & 0 & 2 & 0 \\
15 & $\ZZ_2^4$ & 1 & 1 & 4 & $2^5$ & $2^5$ & $2^5$ & 3 & 1 & 0 & 5 & 9 & 0 & 0 & 6 & 0 & 6 & 0 & 6 & 0 \\
16 & $\ZZ_2^4$ & 1 & 1 & 4 & $2^5$ & $2^5$ & $2^5$ & 4 & 2 & 0 & 7 & 12 &  0 & 0 & 7 & 0 & 7 & 0 & 7 & 0 \\
17 & $\ZZ_2^4$ & 1 & 1 & 8 & $2^5$ & $2^5$ & $2^6$ & 2 & 0 & 0 & 3 & 6 & 0 & 0 & 1 & 0 & 1 & 0 & 1 & 0 \\
18 & $\ZZ_2^4$ & 1 & 1 & 8 & $2^5$ & $2^5$ & $2^6$ & 4 & 2 & 0 & 7 & 12 & 0 & 0 & 2 & 0 & 2 & 0 & 2 & 0 \\
19 & $\ZZ_2^4$ & 1 & 2 & 2 & $2^5$ & $2^5$ & $2^5$ & 2 & 0 & 0 & 3 & 6 &  0 & 0 & 9 & 0 & 9 & 9 & 0 & 0 \\
20 & $\ZZ_2^4$ & 1 & 2 & 2 & $2^5$ & $2^5$ & $2^5$ & 3 & 1 & 0 & 5 & 9 & 0 & 0 & 60 & 0 & 60 & 60 & 0 & 0 \\
21 & $\ZZ_2^4$ & 1 & 2 & 2 & $2^5$ & $2^5$ & $2^5$ & 4 & 2 & 0 & 7 & 12 & 0 & 0 & 50 & 0 & 50 & 50 & 0 & 0 \\
22 & $\ZZ_2^4$ & 1 & 2 & 2 & $2^5$ & $2^5$ & $2^5$ & 5 & 3 & 0 & 9 & 15 & 4 & 0 & 8 & 0 & 8 & 8 & 0 & 0 \\
23 & $\ZZ_2^4$ & 1 & 2 & 4 & $2^5$ & $2^5$ & $2^6$ & 2 & 0 & 0 & 3 & 6 & 0 & 0 & 4 & 0 & 4 & 3 & 1 & 0 \\
24 & $\ZZ_2^4$ & 1 & 2 & 4 & $2^5$ & $2^5$ & $2^6$ & 3 & 1 & 0 & 5 & 9 & 0 & 0 & 10 & 0 & 10 & 9 & 1 & 0 \\
25 & $\ZZ_2^4$ & 1 & 2 & 4 & $2^5$ & $2^5$ & $2^6$ & 4 & 2 & 0 & 7 & 12 & 0 & 0 & 7 & 0 & 7 & 6 & 1 & 0 \\
26 & $\ZZ_2^4$ & 1 & 2 & 4 & $2^5$ & $2^5$ & $2^6$ & 5 & 3 & 0 & 9 & 15 & 0 & 0 & 3 & 0 & 3 & 3 & 0 & 0 \\
27 & $\ZZ_2^4$ & 1 & 2 & 4 & $2^5$ & $2^6$ & $2^5$ & 2 & 0 & 0 & 3 & 6 & 0 & 0 & 3 & 0 & 3 & 0 & 3 & 0 \\
28 & $\ZZ_2^4$ & 1 & 2 & 4 & $2^5$ & $2^6$ & $2^5$ & 3 & 1 & 0 & 5 & 9 & 0 & 0 & 6 & 0 & 6 & 0 & 6 & 0 \\
29 & $\ZZ_2^4$ & 1 & 2 & 4 & $2^5$ & $2^6$ & $2^5$ & 4 & 2 & 0 & 7 & 12 & 0 & 0 & 4 & 0 & 4 & 0 & 4 & 0 \\
30 & $\ZZ_2^4$ & 1 & 2 & 4 & $2^5$ & $2^6$ & $2^5$ & 5 & 3 & 0 & 9 & 15 & 0 & 0 & 1 & 0 & 1 & 0 & 1 & 0 \\
31 & $\ZZ_2^4$ & 1 & 2 & 4 & $2^6$ & $2^5$ & $2^5$ & 2 & 0 & 0 & 3 & 6 & 0 & 0 & 2 & 0 & 2 & 0 & 2 & 0 \\
32 & $\ZZ_2^4$ & 1 & 2 & 4 & $2^6$ & $2^5$ & $2^5$ & 3 & 1 & 0 & 5 & 9 & 0 & 0 & 7 & 0 & 7 & 0 & 7 & 0 \\
33 & $\ZZ_2^4$ & 1 & 2 & 4 & $2^6$ & $2^5$ & $2^5$ & 4 & 2 & 0 & 7 & 12 & 0 & 0 & 7 & 0 & 7 & 0 & 7 & 0 \\
34 & $\ZZ_2^4$ & 1 & 2 & 4 & $2^6$ & $2^5$ & $2^5$ & 5 & 3 & 0 & 9 & 15 & 0 & 0 & 1 & 0 & 1 & 0 & 1 & 0 \\
35 & $\ZZ_2^4$ & 1 & 2 & 8 & $2^5$ & $2^6$ & $2^6$ & 2 & 0 & 0 & 3 & 6 & 0 & 0 & 1 & 0 & 1 & 0 & 1 & 0 \\
36 & $\ZZ_2^4$ & 1 & 2 & 8 & $2^5$ & $2^6$ & $2^6$ & 4 & 2 & 0 & 7 & 12 & 0 & 0 & 1 & 0 & 1 & 0 & 1 & 0 \\
37 & $\ZZ_2^4$ & 1 & 2 & 8 & $2^6$ & $2^5$ & $2^6$ & 2 & 0 & 0 & 3 & 6 & 0 & 0 & 1 & 0 & 1 & 0 & 1 & 0 \\
38 & $\ZZ_2^4$ & 1 & 2 & 8 & $2^6$ & $2^5$ & $2^6$ & 4 & 2 & 0 & 7 & 12 & 0 & 0 & 1 & 0 & 1 & 0 & 1 & 0 \\
39 & $\ZZ_2^4$ & 1 & 4 & 4 & $2^5$ & $2^6$ & $2^6$ & 4 & 2 & 0 & 7 & 12 & 0 & 0 & 1 & 0 & 1 & 0 & 1 & 0 \\
40 & $\ZZ_2^4$ & 2 & 2 & 2 & $2^5$ & $2^5$ & $2^6$ & 2 & 0 & 0 & 3 & 6 & 0 & 0 & 8 & 0 & 8 & 4 & 4 & 0 \\
41 & $\ZZ_2^4$ & 2 & 2 & 2 & $2^5$ & $2^5$ & $2^6$ & 3 & 1 & 0 & 5 & 9 & 0 & 0 & 37 & 0 & 37 & 29 & 8 & 0 \\
42 & $\ZZ_2^4$ & 2 & 2 & 2 & $2^5$ & $2^5$ & $2^6$ & 4 & 2 & 0 & 7 & 12 & 0 & 0 & 35 & 0 & 35 & 31 & 4 & 0 \\
43 & $\ZZ_2^4$ & 2 & 2 & 2 & $2^5$ & $2^5$ & $2^6$ & 5 & 3 & 0 & 9 & 15 & 0 & 0 & 6 & 0 & 6 & 4 & 2 & 0 \\
44 & $\ZZ_2^4$ & 2 & 2 & 4 & $2^5$ & $2^6$ & $2^6$ & 2 & 0 & 0 & 3 & 6 & 0 & 0 & 2 & 0 & 2 & 0 & 2 & 0 \\
45 & $\ZZ_2^4$ & 2 & 2 & 4 & $2^5$ & $2^6$ & $2^6$ & 3 & 1 & 0 & 5 & 9 & 0 & 0 & 1 & 0 & 1 & 0 & 1 & 0 \\
46 & $\ZZ_2^4$ & 2 & 2 & 4 & $2^5$ & $2^6$ & $2^6$ & 4 & 2 & 0 & 7 & 12 & 0 & 0 & 6 & 0 & 6 & 0 & 6 & 0 \\
47 & $\ZZ_2^4$ & 2 & 2 & 4 & $2^5$ & $2^6$ & $2^6$ & 5 & 3 & 0 & 9 & 15 & 0 & 0 & 4 & 0 & 4 & 0 & 4 & 0 \\
\end{longtable}

\begin{remark}\label{Remark-Table}
In this remark we comment on some of the examples in our table.
\begin{enumerate}
    \item 
First, we observe that our birationality criteria can detect the birationality or non-birationality of the canonical map and of the bicanonical map in all of the examples, except for row No.~13: here, $g(C_3)=2$ and Proposition~\ref{hyperellBir} implies that the bicanonical map is not birational. The same also occurs in other rows, for example No.~7: here, we again have $g(C_3)=2$, which explains the failure of birationality for the bicanonical map.

\item 
In contrast to the phenomenon occurring in (1), row No.~23 provides a case in which the bicanonical map is non-birational, but for an \emph{exotic} reason. Here, the genera of the curves are
\[
g(C_1)=5, \qquad g(C_2)=3, \qquad g(C_3)=3.
\]
Hence, the natural morphisms
\[
f_i \colon X \to C_j/G_j \times C_k/G_k \simeq \mathbb P^1 \times \mathbb P^1
\]
with general fibre $C_i$ are not genus-$2$ fibrations. Furthermore, Proposition \ref{Proposition-Fibrations} shows that $X$ does not carry any other fibration onto a surface since $h^{2,0}(X)=h^{1,0}(X)=0$ and $h^{1,1}(X)=3$.  Thus, this phenomenon closely resembles the so-called \emph{non-standard case} for surfaces, see~\cite[Section 2]{BCP06} for an overview on the topic. Note that rows No.  40 and 44 provide exotic examples as well.

\item Finally, we observe that $K_X^{\otimes 2}$ is always base-point free and pose the question of whether this property holds in general for regular unmixed threefolds isogenous to a product. Furthermore, our birationality criteria are always satisfied for the 3-canonical map, even though this information is not included in the table. 
This motivates the question of whether $K_X^{\otimes 3}$ defines a birational map for all regular unmixed threefolds of this type. Recall that for a surface $S$ isogenous to a product $K_S^{\otimes 2}$ is base-point free by Reider's theorem \cite{Reider88} since $K^2_S\geq 8$; moreover, $K_S^{\otimes 3}$ defines an embedding (cf.~Proposition \ref{embedding-SIP}).
\end{enumerate}
\end{remark}

\begin{proposition}\label{Proposition-Fibrations}
    Let $X=(C_1\times C_2\times C_3)/G$ be a threefold isogenous to a product of unmixed type such that 
    $$h^{2,0}(X)=h^{1,0}(X)=0 \quad \makebox{and} \quad h^{1,1}(X)=3.$$
    Then, up to isomorphism, the only  fibrations of $X$ onto surfaces or curves are
    \[
    f_i\colon X \to \prod_{j\neq i} C_j/G\cong \PP^1\times \PP^1, \quad 
    \pi_i\colon X \to C_i/G\cong \PP^1, \quad \makebox{$i=1,2,3.$}
    \]
    Moreover, the $f_i$'s correspond to the two-dimensional faces of $\mathrm{Nef}(X)$, while the $\pi_i's$ correspond to its one-dimensional faces, according to the Picard rank of the target.
\end{proposition}

\begin{proof}
    By using the exponential sequence, the conditions $h^{2,0}(X)=h^{1,0}(X)=0$ imply that
    \[
    H^1(X,\mathcal O_X^{\ast}) \cong H^2(X,\mathbb Z).
    \]
    Moreover, since $h^{1,1}(X)=3$, the threefold $X$ has Picard rank $\rho(X)=3$. Hence, the general fibres $F_i \cong C_j \times C_k$ of the fibrations
    $\pi_i \colon X \to C_i/G$
    generate $N^1(X)_{\RR}$.
    \begin{claim}\label{claim-nef-cone}
     The general fibres $F_i$ of the fibrations $ \pi_i \colon X \to C_i/G_i$ generate the nef cone 
    $$
    \mathrm{Nef}(X)=\mathbb R_{\geq 0} F_1 + \mathbb R_{\geq 0} F_2 + \mathbb R_{\geq 0} F_3.
    $$
    In particular, every $\QQ$-divisor in $\operatorname{Nef}(X)$ is semiample.
    \end{claim}
    \begin{proof}[Proof of Claim \ref{claim-nef-cone}]
    Let $D=a_1F_1+a_2F_2+a_3F_3$ be a nef divisor and let $C_{ij}:=F_i\cdot F_j$. We have
    \[
    0 \leq D\cdot C_{ij}=a_k(F_k \cdot C_{ij})=a_k(F_i\cdot F_j\cdot F_k).
    \]
    Since $F_i\cdot F_j\cdot F_k>0$, we conclude that $a_k\geq 0$. This shows that $\operatorname{Nef}(X)$ is the cone spanned by the fibers $F_i$'s. The last part of the thesis follows immediately since the fibers $F_i$'s are semiample divisors.
    \end{proof}
    
    Now let $f\colon X\to Y $ be a fibration of $X$, i.e., a surjective morphism onto a normal projective variety  $Y$ such that $f_\ast \mathcal{O}_X=\mathcal{O}_Y$ and  assume $\dim Y \in \{1,2\}$. The condition $f_\ast \mathcal{O}_X=\mathcal{O}_Y$ implies that $f$ is induced by a linear system $|D|$, where  $D:=f^\ast A$ and $A$ is very ample. Since $\dim Y >0$, the morphism $f$ is indeed the Iitaka fibration associated with $D$.
    Since $f\colon X\to Y$ is not an isomorphism, we have that $D\in \partial \mathrm{Nef}(X)$, i.e., $D$ belongs to one of the six extremal faces of the first octant of $\RR^3\cong N^1(X)_\RR$. Let $\mathcal{F}$ be  the extremal face which contains $D$ in its relative interior, namely $D\in \overset{\circ}{\mathcal{F}} $.
    \begin{claim}\label{claim-contracted-curves}
        Let $C$ be a curve. If $D'\in  \overset{\circ}{\mathcal{F}}$, then
        $D'\cdot C=0$ if and only if $D\cdot C=0.$
    \end{claim}
    \begin{proof}[Proof of Claim \ref{claim-contracted-curves}]
        Assume $D\cdot C=0$. Since $D\in  \overset{\circ}{\mathcal{F}}$, there exists $G\in \overset{\circ}{\mathcal{F}}$ such that
        $$D\in (G,D'):=\{tG + (1-t) D' \mid t \in (0,1)\}.
        $$
        This yields
        \[
        0=D\cdot C=t_0 \ \underset{\geq 0}{\underbrace{(G\cdot C)}} + (1-t_0)\ \underset{\geq 0}{\underbrace{(D'\cdot C)}}, \quad 0<t_0<1,
        \]
        whence $D'\cdot C=G\cdot C=0$. By symmetry, the claim follows.
    \end{proof}

    Note that Claim \ref{claim-nef-cone} implies that every $\QQ$-divisor $D'\in \operatorname{Nef}(X)\setminus \{0\}$ has an associated Iitaka fibration $f'\colon X\to Y'$. This is indeed the map induced by a linear system $|mD'|$ for a suitable $m\in \mathbb N$.
    Assume that $D'\in \overset{\circ}{\mathcal{F}} $.
    Claim \ref{claim-contracted-curves} implies by Rigidity Lemma \cite[Lemma 1.15]{Debarre} that $f'$  is isomorphic to $f\colon X\to Y$, i.e., there exists an isomorphism $g\colon Y\to Y'$ such that $f'=g\circ f$.
    \begin{claim}\label{claim-face-pullback}
    $\mathcal{F}=f^\ast \operatorname{Nef}(Y). $ 
     \end{claim}
    \begin{proof}[Proof of Claim \ref{claim-face-pullback}]

    Since $\operatorname{Nef}(X)\subseteq \RR^3$ is the first octant, it suffices to consider only $\QQ$-divisors.

    Let $G\in \mathcal{F}$ be a  $\QQ$-divisor and let $f_{G}:X\to Z$ be its associated Iitaka fibration. Note that there exist $m\in \mathbb N$ and a very ample divisor $B$ on $Z$ such that $mG=f_G^\ast B$ is an integral divisor and $f_G$ is induced by the linear system $\vert mG \vert$.
    Now let $C$ be a curve. By applying the same argument given in the proof of Claim \ref{claim-contracted-curves}, we have that 
    $D\cdot C=0$ implies $G\cdot C=0$.
    Hence, by Rigidity Lemma 
    \cite[Lemma 1.15]{Debarre}) there is a morphism $h: Y\to Z$ such that $f_{G}=h\circ f$. This implies that 
    \[
    mG=f^\ast_{G}  B =f^\ast (h^\ast B)\in f^\ast \operatorname{Nef}(Y), 
    \]
    whence $G\in f^\ast \operatorname{Nef}(Y)$.

    Conversely, let $G\in f^\ast\operatorname{Nef}(Y)$ be a $\QQ$-divisor. By projection formula, $G\cdot C=0$ for each curve $C$ contracted by $f$, i.e., $D\cdot C=0$.   Suppose by contradiction $G\notin \mathcal{F}$. Since $f\colon X\to Y$ is not an isomorphism, it contracts some curves, hence $G\notin \operatorname{Amp}(X)$. This implies that there exists an extremal face $\mathcal{F}'\neq \mathcal{F}$ such that $G\in \mathcal{F}'$. Suppose first $\dim \mathcal{F}=2$ and assume w.l.o.g. $\mathcal{F}=\RR_{\ge 0}F_1+\RR_{\ge 0} F_2$. Hence, it must occur
    \[
    \mathcal{F}'\in \{\RR_{\ge 0} F_3, \quad \RR_{\ge 0} F_2+\RR_{\ge 0} F_3, \quad \RR_{\ge 0} F_1+\RR_{\ge 0} F_3\}.
    \]
    Let $C':=F_1\cdot F_2$. One can easily verify that  in each of these cases $D\cdot C'=0$, but $G\cdot C'>0$, a contradiction. Suppose finally $\dim \mathcal{F}=1$ and assume w.l.o.g. $\mathcal{F}=\RR_{\ge 0} F_1$. Hence, it must occur
    \[
    \mathcal{F}'\in \{
    \RR_{\ge 0} F_2,
    \quad
    \RR_{\ge 0} F_1+\RR_{\ge 0} F_2,
    \quad
    \RR_{\ge 0} F_2+\RR_{\ge 0} F_3, \quad 
    \RR_{\ge 0} F_3, 
    \quad 
    \RR_{\ge 0} F_1+\RR_{\ge 0} F_3 \}.
    \]
    Taking respectively $C':=F_1\cdot F_3$ in the first three cases and  $C':=F_1\cdot F_2$ in the last two, we get as before $D\cdot C'=0$, but $G\cdot C'>0$, again a contradiction. Therefore, $G\in\mathcal{F}$.    
     \end{proof}
     
    Since the pullback $f^\ast \colon N^1(Y)_\RR \to N^1(X)_\RR$ is injective and $\mathrm{Nef}(Y)\subseteq N^1(Y)_\RR$ is a full-dimensional closed convex cone, Claim \ref{claim-face-pullback} implies that $\dim \mathcal{F}=\rho(Y).$ 
    
    Finally, since the six fibrations $f_i$'s and $\pi_i$'s  belong to the relative interior of pairwise different faces of the first octant $\operatorname{Nef}(X)\subseteq \RR^3$, the given fibration $f:X\to Y$ coincides with one of those. 
\end{proof}

\end{document}